\documentclass[12pt,thmsa,a4paper]{article}
\usepackage{amssymb}
\usepackage{amsfonts}
\usepackage{amsmath}
\usepackage{float}
\usepackage[colorlinks=true]{hyperref}

\hypersetup{urlcolor=blue, citecolor=blue}
\newtheorem{theorem}{Theorem}

\newtheorem{corollary}{Corollary}

\newtheorem{example}{Example}

\newtheorem{lemma}{Lemma}

\newtheorem{remark}{Remark}

\newenvironment{proof}[1][Proof]{\noindent\textbf{#1.} }{\ \rule{0.5em}{0.5em}}

\usepackage{graphicx}
\graphicspath{{converted_graphics/}}
\begin{document}

\title{Bifurcation Curves in Semipositone Problems with Geometrically
Concave and Concave Nonlinearities\thanks{%
2010 Mathematics Subject Classification: 34B15, 34B18, 34C23, 74G35.}}
\author{Shao-Yuan Huang\thanks{%
Department of Mathematics and Information Education, National Taipei
University of Education, Taipei 106, Taiwan. \ E-mail address:\textit{\ }%
syhuang@mail.ntue.edu.tw}}
\date{}
\maketitle

\begin{abstract}
In this paper, we study the exact multiplicity and bifurcation curves of
positive solutions for the semipositone problem%
\begin{equation*}
\left\{ 
\begin{array}{l}
-u^{\prime \prime }=\lambda f(u),\text{ \ in }\left( -1,1\right) , \\ 
u(-1)=u(1)=0,%
\end{array}%
\right.
\end{equation*}%
where $\lambda >0,$ $f\in C^{2}(0,\infty )$ and there exist $0<\beta
_{1}<\beta _{2}\leq \infty $ such that $f(u)>0$ on $\left( \beta _{1},\beta
_{2}\right) $, and $f(u)<0$ on $\left( 0,\beta _{1}\right) \cup \left( \beta
_{2},\infty \right) $. Note that we allow $f(0^{+})=-\infty $ and provide
many examples to illustrate these results. Furthermore, our results can also
yield the main theorems presented in references \cite{ref8,ref10}.
Additionally, Castro et al. \cite[J. Math. Anal. Appl.]{ref8} claim to have
resolved this issue if $f(0^{+})>-\infty $ and $f$ is concave on $\left(
0,\infty \right) $. However, we find that their proof is incorrect.
Nonetheless, our results demonstrate the correctness of their result.
\end{abstract}

\section{Introduction}

In this paper, we study the exact multiplicity and bifurcation curves of
positive solutions for the semipositone problem%
\begin{equation}
\left\{ 
\begin{array}{l}
-u^{\prime \prime }=\lambda f(u),\text{ \ in }\left( -1,1\right) , \\ 
u(-1)=u(1)=0,%
\end{array}%
\right.  \label{eq1}
\end{equation}%
where $\lambda >0$ is a bifurcation parameter and $f\in C^{2}(0,\infty )$
satisfies the following conditions:

\begin{itemize}
\item[(P$_{1}$)] there exist $0<\beta _{1}<\beta _{2}\leq \infty $ such that 
$f(u)>0$ on $\left( \beta _{1},\beta _{2}\right) $, and $f(u)<0$ on $\left(
0,\beta _{1}\right) \cup \left( \beta _{2},\infty \right) ;$

\item[(P$_{2}$)] $F(u)\equiv \int_{0}^{u}f(t)dt$ exists for $u>0$, and there
exists $\eta \in \left( \beta _{1},\beta _{2}\right) $ such that $F(\eta )=0$%
.
\end{itemize}

\begin{figure}[h] 
  \centering
  \includegraphics[width=6.12in,height=2.87in,keepaspectratio]{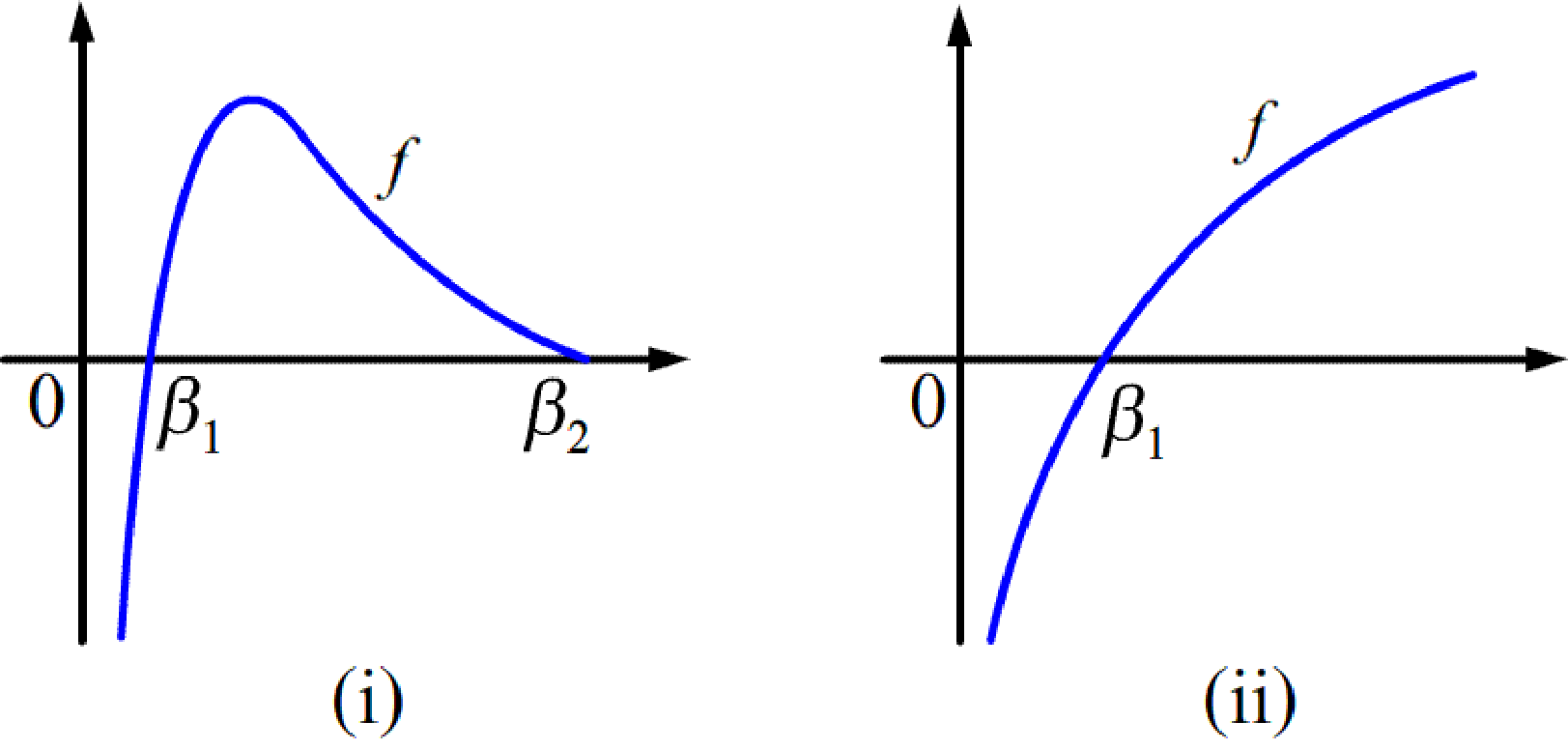}
  \caption{The graph of $f$. (i) $\protect%
\beta _{2}<\infty $. (ii) $\protect\beta _{2}=\infty $.}
  \label{fig1}
\end{figure} 
\noindent Note that we allow $%
f(0^{+})=-\infty $. It is well-known that studying the exact multiplicity of
positive solutions of (\ref{eq1}) is equivalent to analyzing the shape of
the bifurcation curve $S$ of (\ref{eq1}) where%
\begin{equation}
S\equiv \left\{ (\lambda ,\left\Vert u_{\lambda }\right\Vert _{\infty
}):\lambda >0\text{ and }u_{\lambda }\text{ is a positive solution of (\ref%
{eq1})}\right\} \text{.}  \label{S}
\end{equation}%
Therefore, in this paper, we focus on studying the exact shape of the
bifurcation curve $S$. Following the same argument as in \cite{ref15}, the
bifurcation curve $S$ is continuous and consists of exactly one curve, see
also \cite{Hung2}. These proofs are straightforward but tedious, so we omit
them.

The issue of semipositone arises in several fields within applied math and
physics. This includes such as how mechanical structures bend, how bridges
are built to hang, how chemicals react, and ways to model populations that
take into account harvesting, as mentioned in references \cite%
{ref3,ref16,ref23,ref28,ref29}. It's important to note that equation (\ref%
{eq1}) may have nonnegative solutions with zeros in some parts, c.f. \cite%
{ref10}. However, this study will focus on the solutions of (\ref{eq1}) that
are positive.

In 1988, Castro and Shivaji \cite{ref10} were the first to formally
investigate semipositone problems. Generally, it is more challenging to
determine positive solutions for semipositone problems than for positone
problems. This is because, in semipositone cases, the positive solutions
need to work in areas where the nonlinear part can be both negative and
positive. Additionally, this is a significant area of research, with
extensive studies on one-dimensional semipositone problems, as described in
references \cite{ref1,ref5,ref8,ref10,ref15,ref24,ref30,ref31,ref34}. In
particular, the case where $f$ is concave was studied in \cite{ref8}, the
case where $f$ is convex was studied in \cite{ref10}, the case where $f$ is
convex-concave was studied in \cite{ref30}, and the case where $f$ is
concave-convex was also studied in \cite{ref10}.

Next, we mention two special cases. Castro and Shivaji \cite{ref10} studied
the case that $f^{\prime \prime }(u)>0$ for $u>0$, and further obtained the
following result:

\begin{theorem}[{\protect\cite[Theorem 1.1]{ref10}}]
\label{RT1}Consider (\ref{eq1}) with $\beta _{2}=\infty $. Assume that 
\begin{equation}
-\infty <f(0^{+})<0\text{, \ }\lim_{u\rightarrow \infty }\frac{f(u)}{u}%
=\infty \text{, \ }f^{\prime }(u)>0\text{ and }f^{\prime \prime }(u)>0\text{
for }u>0\text{.}  \label{h1}
\end{equation}%
Then there exists $\lambda ^{\ast }>0$ such that the bifurcation curve $S$
is monotone decreasing, starts from $(\lambda ^{\ast },\eta )$ and goes to $%
\left( 0,\infty \right) $, see Figure \ref{RT2}.

\begin{figure}[h] 
  \centering
  \includegraphics[width=3.45in,height=3.34in,keepaspectratio]{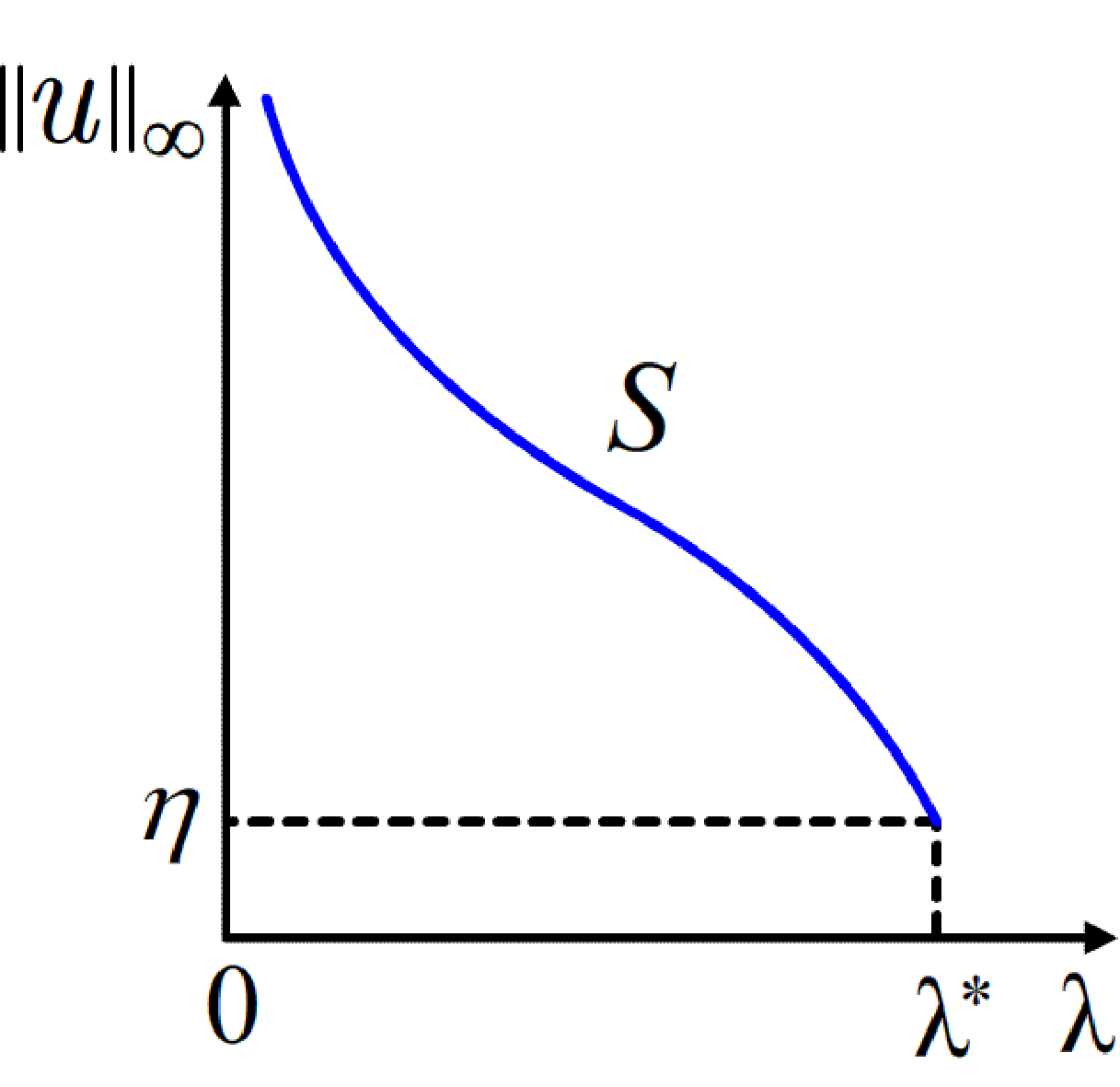}
  \caption{The graphs of the bifurcation curve $S$ of (\protect\ref%
{eq1}) with $\protect\beta _{2}=\infty $ if (\protect\ref{h1}) holds.}
  \label{figCT1}
\end{figure}
\end{theorem}

In 1994, Wang \cite{ref34} studied the case in which $f^{\prime \prime
}(u)<0 $ for $u>0$, and further proved that the bifurcation curve $S$ of (%
\ref{eq1}) with $f(0^{+})>-\infty $ can take one of three possible shapes:
monotone increasing, monotone decreasing, or $\subset $-shaped. However, in
2000, Castro et al. \cite{ref8} improved Wang's results \cite{ref34} and
further pointed out that the bifurcation curve $S$ cannot be monotone
increasing, see Theorem \ref{RT2}.

\begin{theorem}[{\protect\cite[Theorems 1 and 2]{ref8}}]
\label{RT2}Consider (\ref{eq1}) with $\beta _{2}=\infty $. Assume that%
\begin{equation*}
-\infty <f(0^{+})<0\text{, \ }f^{\prime }(u)\geq 0\text{ and }f^{\prime
\prime }(u)<0\text{ for }u>0\text{.}
\end{equation*}
Then the following statements (i)--(iii) hold:

\begin{itemize}
\item[(i)] Assume that $\lim\limits_{u\rightarrow \infty }f^{\prime }(u)>0$
and 
\begin{equation}
\left( \frac{f(u)}{u}\right) ^{\prime }>0\text{ \ for all }u>0.  \label{b1}
\end{equation}%
Then there exist $\lambda ^{\ast }>\lambda _{\ast }>0$ such that the
bifurcation curve $S$ is monotone decreasing, starts from $(\lambda ^{\ast
},\eta )$ and goes to $\left( \lambda _{\ast },\infty \right) $, see Figure %
\ref{fig1-c}(i).

\item[(ii)] Assume that $\lim\limits_{u\rightarrow \infty }f^{\prime }(u)>0$
and there exists $\gamma >0$ such that 
\begin{equation}
\gamma f^{\prime }(\gamma )=f(\gamma )\text{ \ and \ }\left( \frac{f(u)}{u}%
\right) ^{\prime }\left\{ 
\begin{array}{ll}
>0 & \text{for }u\in (0,\gamma )\text{,} \\ 
<0 & \text{for }u\in (\gamma ,\infty ).%
\end{array}%
\right.  \label{b2}
\end{equation}%
Then there exist $\lambda ^{\ast }>0$ and $\lambda _{\ast }>0$ such that the
bifurcation curve $S$ is $\subset $-shaped, starts from $(\lambda ^{\ast
},\eta )$ and goes to $\left( \lambda _{\ast },\infty \right) $, see Figure %
\ref{fig1-c}(ii).

\item[(iii)] Assume that $\lim\limits_{u\rightarrow \infty }f^{\prime }(u)=0$%
. Then there exists $\lambda ^{\ast }>0$ such that the bifurcation curve $S$
is $\subset $-shaped, starts from $(\lambda ^{\ast },\eta )$ and goes to $%
\left( \infty ,\infty \right) $, see Figure \ref{fig1-c}(iii).
\end{itemize}
\end{theorem}

\begin{figure}[h] 
  \centering
  \includegraphics[width=6.26in,height=2.47in,keepaspectratio]{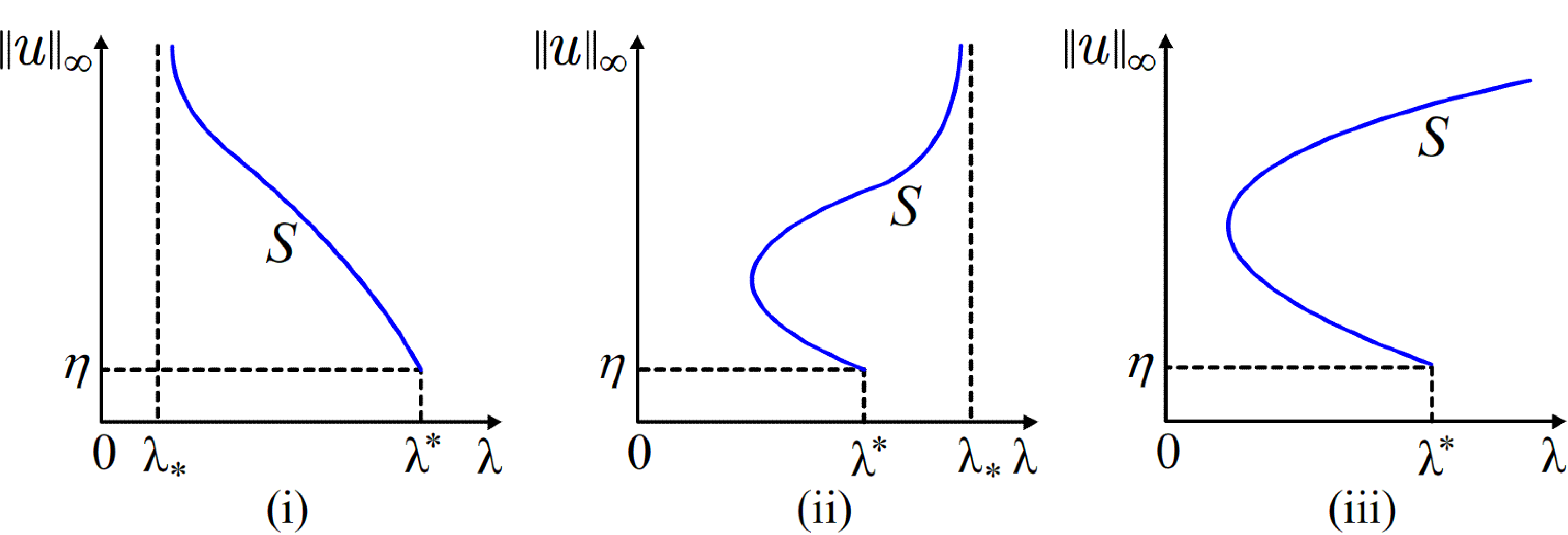}
  \caption{Graphs of bifurcation curves $S$
of (\protect\ref{eq1}) with $f(0^{+})>-\infty $.}
  \label{fig1-c}
\end{figure}

In this paper, we establish certain conditions under which the exact shapes
of the bifurcation curve $S$ can be obtained. As a result, the exact
multiplicity of positive solutions can also be determined. It is worth
mentioning that we demonstrate that our results serve as generalizations of
Theorems \ref{RT1} and \ref{RT2} for when $f$ is convex and concave,
respectively, see Corollary \ref{C2} and Remark \ref{R2} stated below. In
addition, upon careful examination, we have identified the mistake in the
proofs of Theorem \ref{RT2} as presented in \cite[Theorems 1 and 2]{ref8}.
Section 6 provides a detailed discussion on this mistake. Finally, Section 3
provides many examples to illustrate our results.

\section{Main Results}

In this section, we present our main results.

\begin{theorem}
\label{T1}Consider (\ref{eq1}). Let%
\begin{equation}
g(u)\equiv \frac{f(u)}{u}\text{ \ and \ }\hat{\lambda}\equiv \frac{1}{2}%
\left( \int_{0}^{\eta }\frac{1}{\sqrt{-F(u)}}du\right) ^{2}.  \label{L}
\end{equation}%
Then there exists $\kappa \in \lbrack 0,\infty ]$ such that the bifurcation
curve $S$ starts from $(\hat{\lambda},\eta )$ and goes to $\left( \kappa
,\beta _{2}\right) $. Furthermore,

\begin{itemize}
\item[(i)] 
\begin{equation}
\left\{ 
\begin{array}{ll}
\hat{\lambda}\in \left( 0,\infty \right) & \text{if }\lim\limits_{u%
\rightarrow 0^{+}}u^{p}g(u)\in \lbrack -\infty ,0)\text{ for some }p\in
(0,1),\smallskip \\ 
\hat{\lambda}=\infty & \text{if }\lim\limits_{u\rightarrow 0^{+}}g(u)\in
(-\infty ,0]\text{.}%
\end{array}%
\right.  \label{La}
\end{equation}

\item[(ii)] If $\beta _{2}<\infty $, then%
\begin{equation}
\left\{ 
\begin{array}{ll}
\kappa \in (0,\infty ) & \text{if }\lim\limits_{u\rightarrow \beta _{2}^{-}}%
\frac{f(u)}{\left( \beta _{2}-u\right) \left[ -\ln \left( \beta
_{2}-u\right) \right] ^{\tau }}\in (0,\infty ]\text{ for some }\tau >2\text{,%
}\smallskip \\ 
\kappa =\infty & \text{if }\lim\limits_{u\rightarrow \beta _{2}^{-}}\frac{%
f(u)}{\beta _{2}-u}\in \lbrack 0,\infty ).%
\end{array}%
\right.  \label{Lb}
\end{equation}

\item[(iii)] If $\beta _{2}=\infty $, then%
\begin{equation}
\left\{ 
\begin{array}{ll}
\kappa =\infty & \text{if }\lim\limits_{u\rightarrow \infty }g(u)=0\text{,}%
\smallskip \\ 
\kappa \in (0,\infty ) & \text{if }\lim\limits_{u\rightarrow \infty }g(u)\in
(0,\infty )\text{,}\smallskip \\ 
\kappa =0 & \text{if }\lim\limits_{u\rightarrow \infty }g(u)=\infty \text{.}%
\end{array}%
\right.  \label{Lc}
\end{equation}
\end{itemize}
\end{theorem}

\begin{remark}
In Theorem \ref{T1}, if the starting point $(\hat{\lambda},\eta )$ is
finite, then it lies on the bifurcation curve $S.$ Similarly, if the ending
point $\left( \kappa ,\beta _{2}\right) $ is finite, then it also lies on
the bifurcation curve $S$, cf. \cite{AW, Hung2, Hung3, ref31}.
\end{remark}

In order to obtain the exact shape of the bifurcation curve $S$, we set the
following conditions:

\begin{description}
\item[(H$_{1}$)] $g^{\prime }(u)>0$ for $0<u<\beta _{2}.$

\item[(H$_{2}$)] There exists $\sigma >0$ such that $g^{\prime }(u)>0$ for $%
0<u<\sigma $, $g^{\prime }(\sigma )=0$ and $g^{\prime }(u)<0$ for $\sigma
<u<\beta _{2}.$

\item[(H$_{3}$)] $f$ is a geometrically concave function on $\left( \sigma
,\beta _{2}\right) $, that is, 
\begin{equation*}
f(\sqrt{ab})\geq \sqrt{f(a)f(b)}\text{ \ for }a,b\in \left( \sigma ,\beta
_{2}\right) \text{.}
\end{equation*}

\item[(H$_{4}$)] $f$ is a concave function on $\left( 0,\beta _{2}\right) $,
that is, $f^{\prime \prime }(u)<0$ for $0<u<\beta _{2}.$
\end{description}

\begin{remark}
\label{R1}

\begin{itemize}
\item[(i)] The conditions (H$_{3}$) and (H$_{4}$) represent two distinct
concepts. The geometrically concave function may have more than one
inflection point. To facilitate the verification of condition (H$_{3}$),
Hung \cite{Hung} proved that 
\begin{equation}
\text{(H}_{3}\text{) holds \ if and only if \ }\left[ \frac{uf^{\prime }(u)}{%
f(u)}\right] ^{\prime }\leq 0\text{ \ for }\sigma <u<\beta _{2}.  \label{h3a}
\end{equation}

\item[(ii)] Assume that (H$_{4}$) holds. By Lemma \ref{L1}(iii) stated
below, we find that either (H$_{1}$) or (H$_{2}$) holds. Furthermore,%
\begin{equation*}
\left\{ 
\begin{array}{l}
\text{(H}_{1}\text{) holds \ if and only if }\lim\limits_{u\rightarrow \beta
_{2}^{-}}\left[ f(u)-uf^{\prime }(u)\right] \leq 0, \\ 
\text{(H}_{2}\text{) holds \ if and only if }\lim\limits_{u\rightarrow \beta
_{2}^{-}}\left[ f(u)-uf^{\prime }(u)\right] >0.%
\end{array}%
\right.
\end{equation*}
\end{itemize}
\end{remark}

\begin{theorem}
\label{T2}Consider (\ref{eq1}). Let%
\begin{equation}
G\equiv \int_{0}^{\eta }\frac{-\eta f(\eta )-2F(u)+uf(u)}{[-F(u)]^{3/2}}du%
\text{.}  \label{G}
\end{equation}%
Then the following statements (i)--(ii) hold.

\begin{itemize}
\item[(i)] Assume that (H$_{1}$) holds. Then the bifurcation curve $S$ is
monotone decreasing.

\item[(ii)] Assume that either ((H$_{2}$) and (H$_{3}$)) or ((H$_{2}$) and (H%
$_{4}$)) holds. Then the bifurcation curve $S$ is $\subset $-shaped if $G<0$%
, and is monotone increasing if $G\geq 0$.
\end{itemize}
\end{theorem}

\begin{remark}
\label{R2}By Theorems \ref{T1} and \ref{T2}, we obtain the exact shape of
the bifurcation curve $S$, including the positions of its starting and
ending points. From these results, we can understand the exact multiplicity
of positive solutions. Since there are many possibilities, we will not list
them all here.
\end{remark}

\begin{corollary}
\label{C0}Consider (\ref{eq1}) with $\beta _{2}<\infty $. Assume that (H$%
_{4} $) holds. Then the bifurcation curve $S$ is $\subset $-shaped if $G<0$,
and is monotone increasing if $G\geq 0$.
\end{corollary}

Since it is not easy to verify the condition $G<0$, we provide a
sufficiently condition for the condition `$G<0$'.

\begin{theorem}
\label{T0}Consider (\ref{eq1}). Then $G<0$ if one of the following
conditions holds:%
\begin{equation*}
\text{"}f(\eta )-\eta f^{\prime }(\eta )\leq 0\text{", \ "}%
\lim\limits_{u\rightarrow 0^{+}}u^{\frac{1}{3}}f(u)>-\infty \text{" \ and \ "%
}f(0^{+})>-\infty \text{".}
\end{equation*}
\end{theorem}

\begin{corollary}
\label{C1}Consider (\ref{eq1}). Assume that $f(0^{+})>-\infty $. Assume that
either ((H$_{2}$) and (H$_{3}$) hold), or ($\beta _{2}<\infty $ and (H$_{4}$%
) holds). Then the bifurcation curve $S$ is $\subset $-shaped.
\end{corollary}

\begin{remark}
\label{R3}Assume that $\beta _{2}=\infty $, $f(0^{+})>-\infty $ and (H$_{4}$%
) holds. By Theorem \ref{T0}, we obtain that $G<0$. From the following
observations, we find that Theorems \ref{T1} and \ref{T2} can derive Theorem %
\ref{RT2}.

\begin{itemize}
\item[(i)] $\lim\limits_{u\rightarrow \infty }f^{\prime }(u)>0$ and (\ref{b1}%
) holds if and only if 
\begin{equation*}
\lim\limits_{u\rightarrow \infty }g(u)\in (0,\infty )\text{ \ and \ (H}_{1}%
\text{) holds.}
\end{equation*}%
So Theorem \ref{RT2}(i) can be obtained by Theorems \ref{T1} and \ref{T2}(i).

\item[(ii)] $\lim\limits_{u\rightarrow \infty }f^{\prime }(u)>0$ and (\ref%
{b2}) holds if and only if 
\begin{equation*}
\lim\limits_{u\rightarrow \infty }g(u)\in (0,\infty )\text{ \ and \ (H}_{2}%
\text{) holds.}
\end{equation*}%
So Theorem \ref{RT2}(ii) can be obtained by Theorems \ref{T1} and \ref{T2}%
(ii).

\item[(iii)] Assume that $\lim\limits_{u\rightarrow \infty }f^{\prime }(u)=0$%
. By Lemma \ref{L1}(iii) stated below, we obtain that%
\begin{equation*}
\lim\limits_{u\rightarrow \infty }g(u)=0\text{ \ and \ (H}_{4}\text{) holds.}
\end{equation*}%
So Theorem \ref{RT2}(iii) can be obtained by Theorems \ref{T1} and \ref{T2}%
(ii).
\end{itemize}
\end{remark}

\begin{remark}
Assume that (H$_{4}$) holds. By Theorem \ref{RT2}, it is impossible that the
corresponding curve is monotone increasing in the case where $%
f(0^{+})>-\infty $. But by Theorem \ref{T2}, it may be possible for $%
f(0^{+})=-\infty $, see Example \ref{E3}.
\end{remark}

\begin{corollary}
\label{C2}Consider (\ref{eq1}) with $\beta _{2}=\infty $. If $f^{\prime
\prime }(u)>0$ on $\left( 0,\infty \right) $, then the bifurcation curve $S$
is monotone decreasing, starts from $(\hat{\lambda},\eta )$ and goes to $%
\left( \kappa ,\infty \right) $ where%
\begin{equation}
\hat{\lambda}\left\{ 
\begin{array}{ll}
=\infty & \text{if }f(0^{+})=0\text{,} \\ 
\in (0,\infty ) & \text{if }f(0^{+})<0,%
\end{array}%
\right. \text{ \ and \ }\kappa \left\{ 
\begin{array}{ll}
=0 & \text{if }\lim\limits_{u\rightarrow \infty }f^{\prime }(u)=\infty \text{%
,} \\ 
\in (0,\infty ) & \text{if }\lim\limits_{u\rightarrow \infty }f^{\prime
}(u)\in (0,\infty ),%
\end{array}%
\right.  \label{c2a}
\end{equation}%
see Figure \ref{figCT2}. 

\begin{figure}[h] 
  \centering
  \includegraphics[width=4.58in,height=4.69in,keepaspectratio]{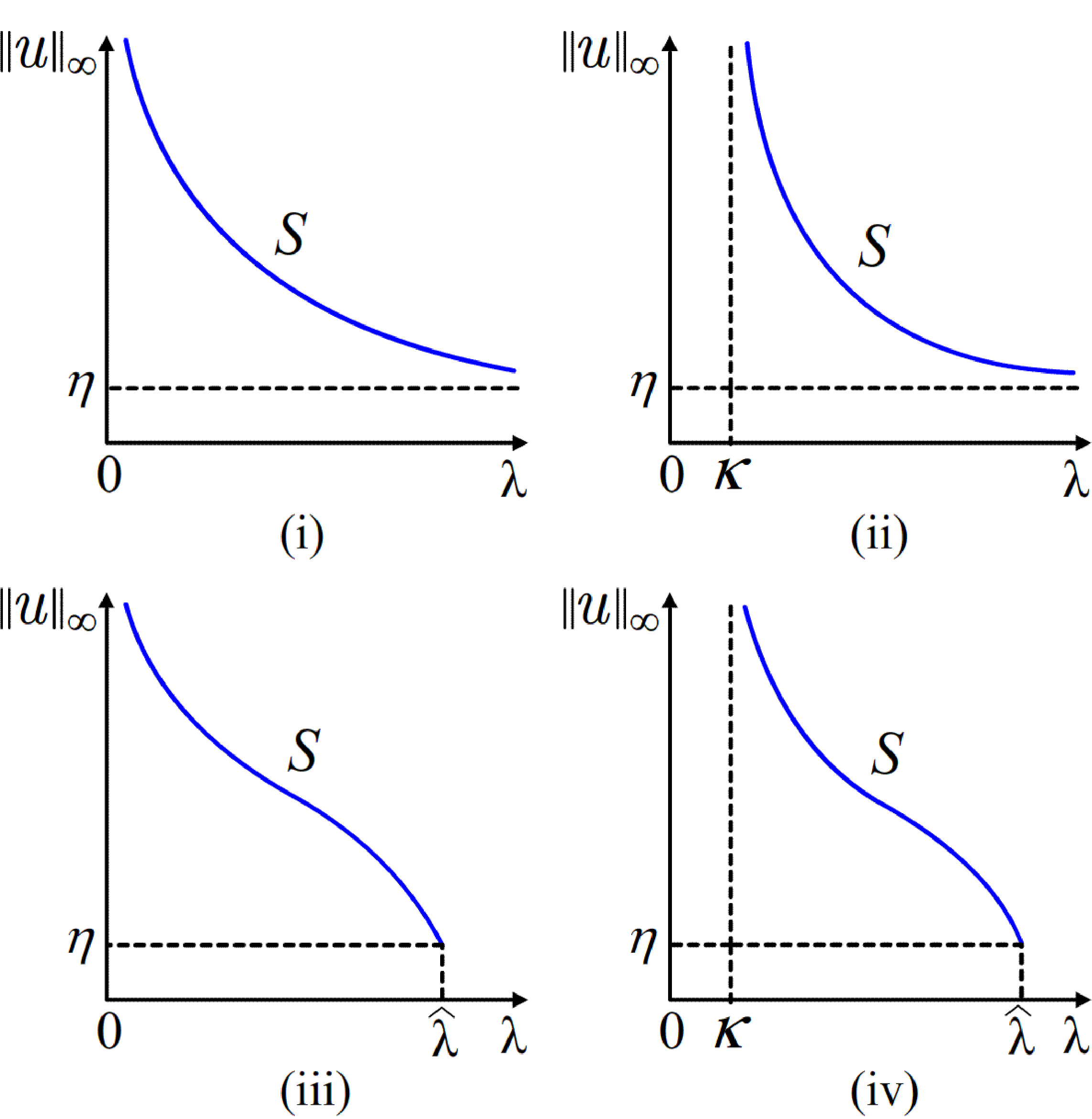}
  \caption{Graphs
of the bifurcation curve $S$ if $f^{\prime \prime }(u)>0$ on $\left(
0,\infty \right) $. (i) $f(0^{+})=0$ and $\lim\limits_{u\rightarrow \infty
}f^{\prime }(u)=\infty $. (ii) $f(0^{+})=0$ and $\lim\limits_{u\rightarrow
\infty }f^{\prime }(u)\in (0,\infty )$. (iii) $f(0^{+})<0$ and $%
\lim\limits_{u\rightarrow \infty }f^{\prime }(u)=\infty $. (iv) $f(0^{+})<0$
and $\lim\limits_{u\rightarrow \infty }f^{\prime }(u)\in (0,\infty ).$}
  \label{figCT2}
\end{figure}
\end{corollary}

\begin{remark}
\label{R4}Assume that $\beta _{2}=\infty $ and $f^{\prime \prime }(u)>0$ on $%
\left( 0,\infty \right) $. Obviously, Corollary \ref{C2} can obtain more
possibilities than Theorem \ref{RT1}. Furthermore, since the condition (\ref%
{h1}) implies that%
\begin{equation*}
f(0^{+})<0\text{ \ and }\lim_{u\rightarrow \infty }f^{\prime }(u)=\infty ,
\end{equation*}%
the shape of the corresponding bifurcation curve in Theorem \ref{RT1}
corresponds to the shape in Figure \ref{figCT2}(iii). Thus Corollary \ref{C2}
is the generalization of Theorem \ref{RT1}.
\end{remark}

\section{Examples}

In this section, we give many examples. Examples \ref{E1}--\ref{E4}
represent singular semipositone problems, and Examples \ref{E5}--\ref{E9}
represent nonsingular semipositone problems.

In Examples \ref{E1}--\ref{E3}, the corresponding nonlinearities $f$ are
concave and $\beta _{2}=\infty $.

\begin{example}
\label{E1}Consider (\ref{eq1}) with $f(u)=\ln u$. Clearly, $f(0^{+})=-\infty 
$ and $\beta _{2}=\infty $. Then the bifurcation curve $S$ is $\subset $%
-shaped, starts from $(\hat{\lambda},e)$ and goes to $\left( \infty ,\infty
\right) $. where $\hat{\lambda}\approx 8.539$.

Next, we prove the above statement. We compute%
\begin{equation*}
f^{\prime \prime }(u)=-\frac{1}{u^{2}}<0\text{ \ for }u>0\text{ \ and \ }%
F(u)=u\ln u-u.
\end{equation*}%
Then (P$_{1}$)--(P$_{2}$) and (H$_{4}$) hold, and we obtain $\eta =e$. By L'H%
\"{o}pital's rule, we compute%
\begin{equation*}
\lim\limits_{u\rightarrow 0^{+}}\sqrt{u}g(u)=\lim\limits_{u\rightarrow 0^{+}}%
\frac{\ln (u)}{\sqrt{u}}=-\infty \text{ \ and \ }\lim\limits_{u\rightarrow
\infty }g(u)=\lim\limits_{u\rightarrow \infty }\frac{1}{u}=0.
\end{equation*}%
So by Theorem \ref{T1}, we obtain 
\begin{equation}
\hat{\lambda}\in (0,\infty )\text{ \ and \ }\kappa =\infty .  \label{E1a}
\end{equation}%
Since 
\begin{equation*}
\lim\limits_{u\rightarrow \infty }\left[ f(u)-uf^{\prime }(u)\right]
=\lim\limits_{u\rightarrow \infty }\left( \ln u-1\right) =\infty >0,
\end{equation*}%
and by Remark \ref{R1}(ii), we see that (H$_{2}$) holds. Since%
\begin{equation*}
\lim_{u\rightarrow 0^{+}}u^{\frac{1}{3}}f(u)=\lim_{u\rightarrow 0^{+}}\left(
u^{\frac{1}{3}}\ln u\right) =0,
\end{equation*}%
and by Theorem \ref{T0}, we obtain that $G<0$. So by Theorems \ref{T1}--\ref%
{T2}(ii) and (\ref{E1a}), the bifurcation curve $S$ is $\subset $-shaped,
starts from $(\hat{\lambda},e)$ and goes to $\left( \infty ,\infty \right) $%
. Finally, we apply Maple software to obtain $\hat{\lambda}\approx 8.539$.
This proof is complete.
\end{example}

\begin{example}
\label{E2}Consider (\ref{eq1}) with $f(u)=\sigma u-u^{-p}$, where $\sigma >0$
and $0<p<1$. Clearly, $f(0^{+})=-\infty $ and $\beta _{2}=\infty $. Then the
bifurcation curve $S$ is monotone decreasing, starts from $(\hat{\lambda}%
,\eta )$ and goes to $\left( \kappa ,\infty \right) $ where 
\begin{equation}
\eta =\left[ \frac{2}{\sigma \left( 1-p\right) }\right] ^{\frac{1}{p+1}}%
\text{ \ and \ }\kappa >0.  \label{E2a}
\end{equation}

Next, we prove the above statement. We compute%
\begin{equation*}
F(u)=\frac{\sigma }{2}u^{1-p}\left[ u^{p+1}-\frac{2}{\sigma \left(
1-p\right) }\right] .
\end{equation*}%
Then (P$_{1}$)--(P$_{2}$) hold, and we obtain $\eta $ defined by (\ref{E2a}%
). Since%
\begin{equation*}
\lim\limits_{u\rightarrow 0^{+}}\sqrt{u}g(u)=\lim\limits_{u\rightarrow 0^{+}}%
\frac{\sigma u-u^{-p}}{\sqrt{u}}=-\infty \text{ \ and \ }\lim\limits_{u%
\rightarrow \infty }g(u)=\sigma \in (0,\infty ),
\end{equation*}%
and by Theorem \ref{T1}, we obtain 
\begin{equation}
\hat{\lambda}\in (0,\infty )\text{ \ and \ }\kappa \in (0,\infty ).
\label{E2b}
\end{equation}%
We compute%
\begin{equation*}
g^{\prime }(u)=\frac{d}{du}\left( \frac{\sigma u^{p+1}-1}{u^{p+1}}\right) =%
\frac{p+1}{u^{p+2}}>0,
\end{equation*}%
which implies that (H$_{1}$) holds. So by Theorems \ref{T1}--\ref{T2}(i) and
(\ref{E2b}), the bifurcation curve $S$ is monotone decreasing, starts from $(%
\hat{\lambda},\eta )$ and goes to $\left( \kappa ,\infty \right) $. This
proof is complete.
\end{example}

\begin{example}
\label{E3}Consider (\ref{eq1}) with $f(u)=\sigma -1/\sqrt{u}$, where $\sigma
>0$. Clearly, $f(0^{+})=-\infty $ and $\beta _{2}=\infty $. Then the
bifurcation curve $S$ is monotone increasing, starts from $(2\pi ^{2}/\sigma
^{3},4/\sigma ^{2})$ and goes to $\left( \infty ,\infty \right) $.

Next, we prove the above statement. We compute%
\begin{equation*}
f^{\prime \prime }(u)=-\frac{3}{4u^{5/2}}<0\text{ \ for }u>0\text{ \ and \ }%
F(u)=\sigma u-2\sqrt{u}.
\end{equation*}%
Then (P$_{1}$)--(P$_{2}$) and (H$_{4}$) hold, and we obtain $\eta =4/\sigma
^{2}$. Since%
\begin{equation*}
\lim\limits_{u\rightarrow 0^{+}}\sqrt{u}g(u)=\lim\limits_{u\rightarrow 0^{+}}%
\frac{\sigma -1/\sqrt{u}}{\sqrt{u}}=-\infty \text{ \ and \ }%
\lim\limits_{u\rightarrow \infty }g(u)=0,
\end{equation*}%
and by Theorem \ref{T1}, we obtain 
\begin{equation}
\hat{\lambda}\in (0,\infty )\text{ \ and \ }\kappa =\infty .  \label{E3a}
\end{equation}%
In fact, we can apply Mathematica software to compute%
\begin{equation}
\int \frac{1}{\sqrt{-F(u)}}du=-\frac{2\sqrt{-F(u)}}{\sigma }-\frac{4}{\sigma
^{3/2}}\arctan \left( \frac{\sqrt{-\sigma F(u)}}{-2+\sqrt{u}\sigma }\right) .
\label{E3c}
\end{equation}%
By L'H\^{o}pital's rule, then%
\begin{equation*}
\lim_{u\rightarrow \eta ^{+}}\arctan \left( \frac{\sqrt{-\sigma F(u)}}{-2+%
\sqrt{u}\sigma }\right) =-\frac{\pi }{2}.
\end{equation*}%
So by (\ref{E3c}), we obtain%
\begin{equation}
\hat{\lambda}=\frac{1}{2}\left( \int_{0}^{\eta }\frac{1}{\sqrt{-F(u)}}%
du\right) ^{2}=\frac{1}{2}\left( \frac{2\pi }{\sigma ^{3/2}}\right) ^{2}=%
\frac{2\pi ^{2}}{\sigma ^{3}}.  \label{E3b}
\end{equation}%
Since 
\begin{equation*}
\lim\limits_{u\rightarrow \infty }\left[ f(u)-uf^{\prime }(u)\right]
=\lim\limits_{u\rightarrow \infty }\left( \sigma -\frac{3}{2\sqrt{u}}\right)
=\sigma >0,
\end{equation*}%
and by Remark \ref{R1}(ii), we see that (H$_{2}$) holds. We compute and find
that%
\begin{eqnarray*}
G &=&\int_{0}^{\eta }\frac{-\eta f(\eta )-2F(u)+uf(u)}{[-F(u)]^{3/2}}%
du=\int_{0}^{\eta }\frac{\left( -\sqrt{u}\sigma +2\right) \left( \sqrt{u}%
\sigma -1\right) }{\sigma \left( -\sigma u+2\sqrt{u}\right) ^{3/2}}du \\
&=&\int_{0}^{\eta }\frac{f(u)}{\sigma \sqrt{-F(u)}}du=\left. \frac{2\sqrt{%
-F(u)}}{-\sigma }\right\vert _{u=0}^{u=\eta }=0.
\end{eqnarray*}%
So by Theorems \ref{T1}--\ref{T2}(ii), (\ref{E3a}) and (\ref{E3b}), the
bifurcation curve $S$ is monotone increasing, starts from $(2\pi ^{2}/\sigma
^{3},4/\sigma ^{2})$ and goes to $\left( \infty ,\infty \right) $. This
proof is complete.
\end{example}

In Example \ref{E4}, the conditions (H$_{2}$) and (H$_{3}$) hold, and $\beta
_{2}<\infty $. In fact, the corresponding nonlinearity $f$ is concave-convex.

\begin{example}
\label{E4}Consider (\ref{eq1}) with 
\begin{equation*}
f(u)=4-\sqrt{u}-\frac{1}{\sqrt{u}}.
\end{equation*}%
Clearly, $f(0^{+})=-\infty $ and $\beta _{2}=7+4\sqrt{3}<\infty $. Then the
bifurcation curve $S$ is monotone increasing, starts from $(\hat{\lambda}%
,\eta )$ and goes to $\left( \infty ,\beta _{2}\right) $ where $\hat{\lambda}%
\approx 0.434$ and $\eta =(3-\sqrt{6})^{2}$.

Next, we prove above statement. It is easy to compute%
\begin{equation*}
F(u)=-\frac{2}{3}\sqrt{u}\left( u-6\sqrt{u}+3\right) ,
\end{equation*}%
\begin{equation*}
\beta _{1}=7-4\sqrt{3}\text{ }\left( \approx 0.071\right) \text{\ \ and \ }%
\beta _{2}=7+4\sqrt{3}\text{ }\left( \approx 13.928\right) .
\end{equation*}%
Then (P$_{1}$)--(P$_{2}$) hold, and we obtain $\eta =(3-\sqrt{6})^{2}$. By
L'H\"{o}pital's rule, we compute%
\begin{equation*}
\lim\limits_{u\rightarrow 0^{+}}\sqrt{u}g(u)=-\infty \text{ \ and \ }%
\lim\limits_{u\rightarrow \beta _{2}^{-}}\frac{f(u)}{\beta _{2}-u}%
=-f^{\prime }(\beta _{2})=\frac{\beta _{2}-1}{2\left( \sqrt{\beta _{2}}%
\right) ^{3}}>0.
\end{equation*}%
So by Theorem \ref{T1}, we obtain%
\begin{equation}
\hat{\lambda}\in (0,\infty )\text{ \ and \ }\kappa =\infty .  \label{E4a}
\end{equation}%
We further compute and find that 
\begin{equation}
g^{\prime }(u)=\frac{u-8\sqrt{u}+3}{2u^{5/2}}\text{ }\left\{ 
\begin{array}{ll}
>0 & \text{for }0<u<\sigma , \\ 
=0 & \text{for }u=\sigma =29-8\sqrt{13}, \\ 
<0 & \text{for }\sigma <u<\beta _{2},%
\end{array}%
\right.  \label{E4b}
\end{equation}%
and%
\begin{equation}
\left[ \frac{uf^{\prime }(u)}{f(u)}\right] ^{\prime }=\frac{\sqrt{u}}{%
u^{2}f^{2}(u)}p_{1}(u),  \label{E4c}
\end{equation}%
where $p_{1}(u)\equiv -u+\sqrt{u}-1$. Since $p_{1}^{\prime }(u)=\left( 1-2%
\sqrt{u}\right) /\left( 2\sqrt{u}\right) ,$ we observe that%
\begin{equation*}
p_{1}(u)<p_{1}(\frac{1}{4})=-\frac{3}{4}<0\text{ for }u>0\text{.}
\end{equation*}%
So (H$_{2}$) and (H$_{3}$) hold by (\ref{E4b}) and (\ref{E4c}). Then we
apply Maple software to compute 
\begin{equation*}
G\approx 0.1497>0\text{ \ and \ }\hat{\lambda}\approx 0.434.
\end{equation*}%
So by Theorems \ref{T1}--\ref{T2}(ii) and (\ref{E4a}), the bifurcation curve 
$S$ is monotone increasing, starts from $(\hat{\lambda},\eta )$ and goes to $%
\left( \infty ,\beta _{2}\right) $. This proof is complete.
\end{example}

In Example \ref{E5}, the corresponding nonlinearities $f$ are concave and $%
\beta _{2}<\infty $.

\begin{example}
\label{E5}Consider (\ref{eq1}) with $f(u)=-(u-a)(u-b)$ where $b>a>0$.
Clearly, $f(0^{+})>-\infty $, $\beta _{1}=a$ and $\beta _{2}=b$. Then the
following statements (i)--(ii) hold.

\begin{itemize}
\item[(i)] If $3a\geq b$, the bifurcation curve $S$ does not exist.

\item[(ii)] If $3a<b$, there the bifurcation curve $S$ is $\subset $-shaped,
starts from $(\hat{\lambda},\eta )$ and goes to $\left( \infty ,b\right) $
where%
\begin{equation}
\hat{\lambda}\in (0,\infty )\text{ \ and \ }\eta =\frac{3}{4}\left(
a+b\right) -\frac{1}{4}\sqrt{3\left( a-3b\right) \left( 3a-b\right) }
\label{E5a}
\end{equation}
\end{itemize}

Next, we prove above statements. We compute%
\begin{equation}
F(u)=u\left( -\frac{u^{2}}{3}+\frac{a+b}{2}u-ab\right) .  \label{E5b}
\end{equation}%
Then we consider two cases.

\textbf{Case 1}. Assume that $3a\geq b$. By the discriminant of quadratic
polynomial, we observe that $F(u)\leq 0$ on $\left( 0,\infty \right) $. It
implies that the statement (i) holds.

\textbf{Case 2.} Assume that $3a<b$. Since $f^{\prime \prime }(u)=-1<0$ for $%
u>0$, we see that (P$_{1}$)--(P$_{2}$) and (H$_{4}$) hold, and we obtain $%
\eta $ defined by (\ref{E5a}). Since 
\begin{equation*}
\lim\limits_{u\rightarrow 0^{+}}u^{1/2}g(u)=\lim\limits_{u\rightarrow 0^{+}}%
\frac{f(u)}{\sqrt{u}}=-\infty \text{ \ and \ }\lim\limits_{u\rightarrow
\beta _{2}^{-}}\frac{f(u)}{\beta _{2}-u}=b-a\in (0,\infty ),
\end{equation*}%
and by Theorem \ref{T1} and Corollary \ref{C1}, the bifurcation curve $S$ is 
$\subset $-shaped, starts from $(\hat{\lambda},\eta )$ and goes to $\left(
\infty ,b\right) $. where $\hat{\lambda}\in (0,\infty )$\ and\ $\kappa
=\infty $. The statement (ii) holds.

\noindent By Cases 1--2, the proof is complete.
\end{example}

In Example \ref{E6}, the corresponding nonlinearities $f$ is convex and $%
\beta _{2}=\infty $.

\begin{example}
\label{E6}Consider (\ref{eq1}) with $f(u)=\exp (u)-c$ where $c>1$. Clearly, $%
f(0^{+})>-\infty $ and $\beta _{2}=\infty $. Then the bifurcation curve $S$
is monotone decreasing, starts from $(\hat{\lambda},\eta )$ and goes to $%
\left( 0,\infty \right) $ where $\hat{\lambda}\in (0,\infty )$ and $\eta $
satisfies $\exp (\eta )-c\eta -1=c$.

Next, we prove above statement. It is easy to see that (P$_{1}$)--(P$_{2}$)
hold. We compute and find that%
\begin{equation*}
f^{\prime \prime }(u)>0\text{ for }u>0\text{, \ }f(0^{+})=1-c<0\text{ \ and
\ }\lim\limits_{u\rightarrow \infty }f^{\prime }(u)=\infty .
\end{equation*}%
So by Theorem \ref{T1} and Corollary \ref{C2}, then the bifurcation curve $S$
is monotone decreasing, starts from $(\hat{\lambda},\eta )$ and goes to $%
\left( 0,\infty \right) $. This proof is complete.
\end{example}

In Examples \ref{E7} and \ref{E8}, the conditions (H$_{2}$) and (H$_{3}$)
hold, and $\beta _{2}<\infty $. In fact, the corresponding nonlinearities $f$
are convex-concave and convave-convex, respectively.

\begin{example}
\label{E7}Consider (\ref{eq1}) with $f(u)=\left( 1-u^{2}\right) \left(
u-3\right) $. Clearly, $f(0^{+})>-\infty $ and $\beta _{2}=3$. Then the
bifurcation curve $S$ is $\subset $-shaped, starts from $(\hat{\lambda},2)$
and goes to $\left( \infty ,3\right) $ where $\hat{\lambda}\approx 3.043$.

Next, we prove above statement. We compute%
\begin{equation*}
F(u)=\frac{1}{4}u\left( u-2\right) \left( -u^{2}+2u+6\right) .
\end{equation*}%
Then (P$_{1}$)--(P$_{2}$) hold, and we obtain $\eta =2$. Since%
\begin{equation*}
\lim\limits_{u\rightarrow 0^{+}}\sqrt{u}g(u)=\lim\limits_{u\rightarrow 0^{+}}%
\frac{f(u)}{\sqrt{u}}=-\infty \text{ \ and \ }\lim\limits_{u\rightarrow
\beta _{2}^{-}}\frac{f(u)}{\beta _{2}-u}=8>0,
\end{equation*}%
and by Theorem \ref{T1}, we obtain 
\begin{equation}
\hat{\lambda}\in (0,\infty )\text{ \ and \ }\kappa =\infty .  \label{E7c}
\end{equation}%
By elementary analysis, there exists $\sigma $ $\left( \approx 1.910\right) $
such that 
\begin{equation*}
g^{\prime }(u)=\frac{-2u^{3}+3u^{2}+3}{u^{2}}\left\{ 
\begin{array}{ll}
>0 & \text{for }0<u<\sigma , \\ 
=0 & \text{for }u=\sigma , \\ 
<0 & \text{for }\sigma <u<\beta _{2}=3,%
\end{array}%
\right.
\end{equation*}%
see Figure \ref{fig2}(i). It implies that (H$_{2}$) holds. We compute and
find that%
\begin{equation*}
\left[ \frac{uf^{\prime }(u)}{f(u)}\right] ^{\prime }=\frac{%
-3u^{4}-4u^{3}+30u^{2}-36u-3}{\left( u-3\right) ^{2}\left( u^{2}-1\right)
^{2}}<0\text{ \ for }\sigma <u<\beta _{2}=3,
\end{equation*}%
see Figure \ref{fig2}(ii). It implies that (H$_{3}$) holds. Since $%
f(0^{+})=-3>-\infty $, and by Theorem \ref{T1}, Corollary \ref{C1} and (\ref%
{E7c}), the bifurcation curve $S$ is $\subset $-shaped, starts from $(\hat{%
\lambda},2)$ and goes to $\left( \infty ,3\right) $. Finally, we apply Maple
software to obtain $\hat{\lambda}\approx 3.043$. This proof is complete. 

\begin{figure}[tbp] 
  \centering
  \includegraphics[width=6.12in,height=2.87in,keepaspectratio]{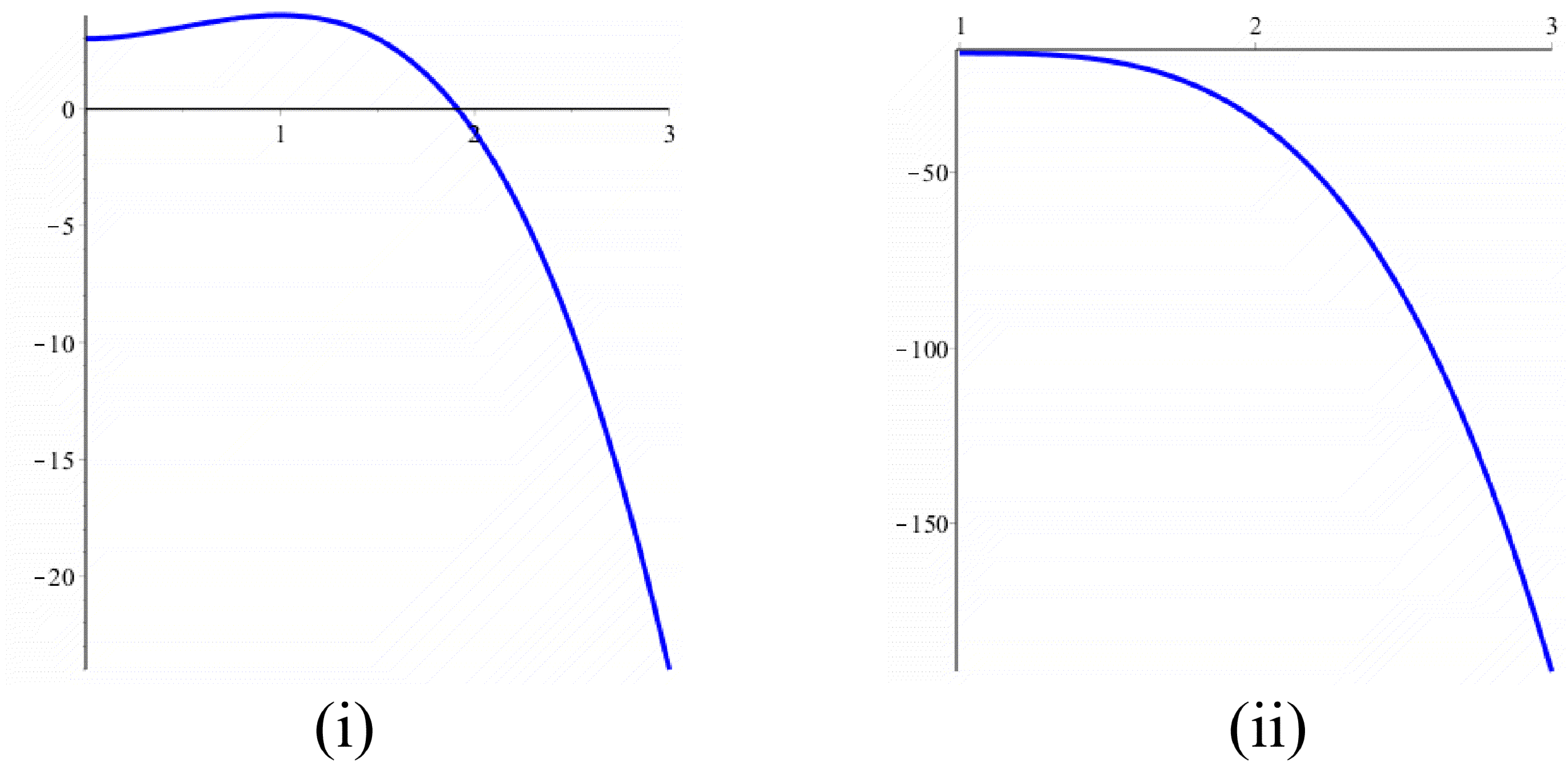}
  \caption{(i) The graph of $%
-2u^{3}+3u^{2}+3 $ on $\left( 0,3\right) $. (ii) The graph of $%
-3u^{4}-4u^{3}+30u^{2}-36u-3$ on $\left( 1,3\right) .$}
  \label{fig2}
\end{figure}

\end{example}

\begin{example}
\label{E8}Consider (\ref{eq1}) with $%
f(u)=-15u^{4}+140u^{3}-450u^{2}+540u-138 $. Clearly, $f(0^{+})>-\infty $ and 
$\beta _{2}<\infty $. Then the bifurcation curve $S$ is $\subset $-shaped,
starts from $(\hat{\lambda},\eta )$ and goes to $\left( \infty ,\beta
_{2}\right) $ where $\hat{\lambda}\approx 0.038$, $\eta \approx 0.814$ and $%
\beta _{2}\approx 2.551$.

Next, we prove above statement. By elementary analysis, we find that (P$_{1}$%
)--(P$_{2}$) hold. Furthermore, we have $\beta _{1}$ $\left( \approx
0.344\right) $ and $\beta _{2}$ $\left( \approx 2.551\right) $. We compute%
\begin{equation*}
F(u)=-u(3u^{4}-35u^{3}+150u^{2}-270u+138).
\end{equation*}%
Then there exists $\eta $ $\left( \approx 0.814\right) >0$ such that $F(\eta
)=0.$ Since%
\begin{equation*}
\lim\limits_{u\rightarrow 0^{+}}\sqrt{u}g(u)=\lim\limits_{u\rightarrow 0^{+}}%
\frac{f(u)}{\sqrt{u}}=-\infty \text{ \ and \ }\lim\limits_{u\rightarrow
\beta _{2}^{-}}\frac{f(u)}{\beta _{2}-u}=K(\beta _{2}-\beta _{1})>0
\end{equation*}%
\ for some $K\geq 0$, and by Theorem \ref{T1}, we obtain 
\begin{equation}
\hat{\lambda}\in (0,\infty )\text{ \ and \ }\kappa =\infty .  \label{E8c}
\end{equation}%
By elementary analysis, there exists $\sigma $ $\left( \approx 0.709\right) $
such that%
\begin{equation*}
g^{\prime }(u)=\frac{-45u^{4}+280u^{3}-450u^{2}+138}{u^{2}}\left\{ 
\begin{array}{ll}
>0 & \text{for }0<u<\sigma , \\ 
=0 & \text{for }u=\sigma , \\ 
<0 & \text{for }\sigma <u<\beta _{2}=3,%
\end{array}%
\right.
\end{equation*}%
see Figure \ref{fig3}(i). It implies that (H$_{2}$) holds. We compute and
find that%
\begin{equation*}
\left[ \frac{uf^{\prime }(u)}{f(u)}\right] ^{\prime }=\frac{%
-60(u-3)(35u^{5}-345u^{4}+1230u^{3}-1902u^{2}+1242u-414)}{f^{2}(u)}<0
\end{equation*}%
for $\sigma <u<\beta _{2}$, see Figure \ref{fig3}(ii). Since $%
f(0^{+})=-138>-\infty $, and by Theorem \ref{T1}, Corollary \ref{C1} and (%
\ref{E8c}), the bifurcation curve $S$ is $\subset $-shaped, starts from $(%
\hat{\lambda},\eta )$ and goes to $\left( \infty ,\beta _{2}\right) $. The
proof is complete. 
\begin{figure}[h] 
  \centering
  \includegraphics[width=6.12in,height=2.87in,keepaspectratio]{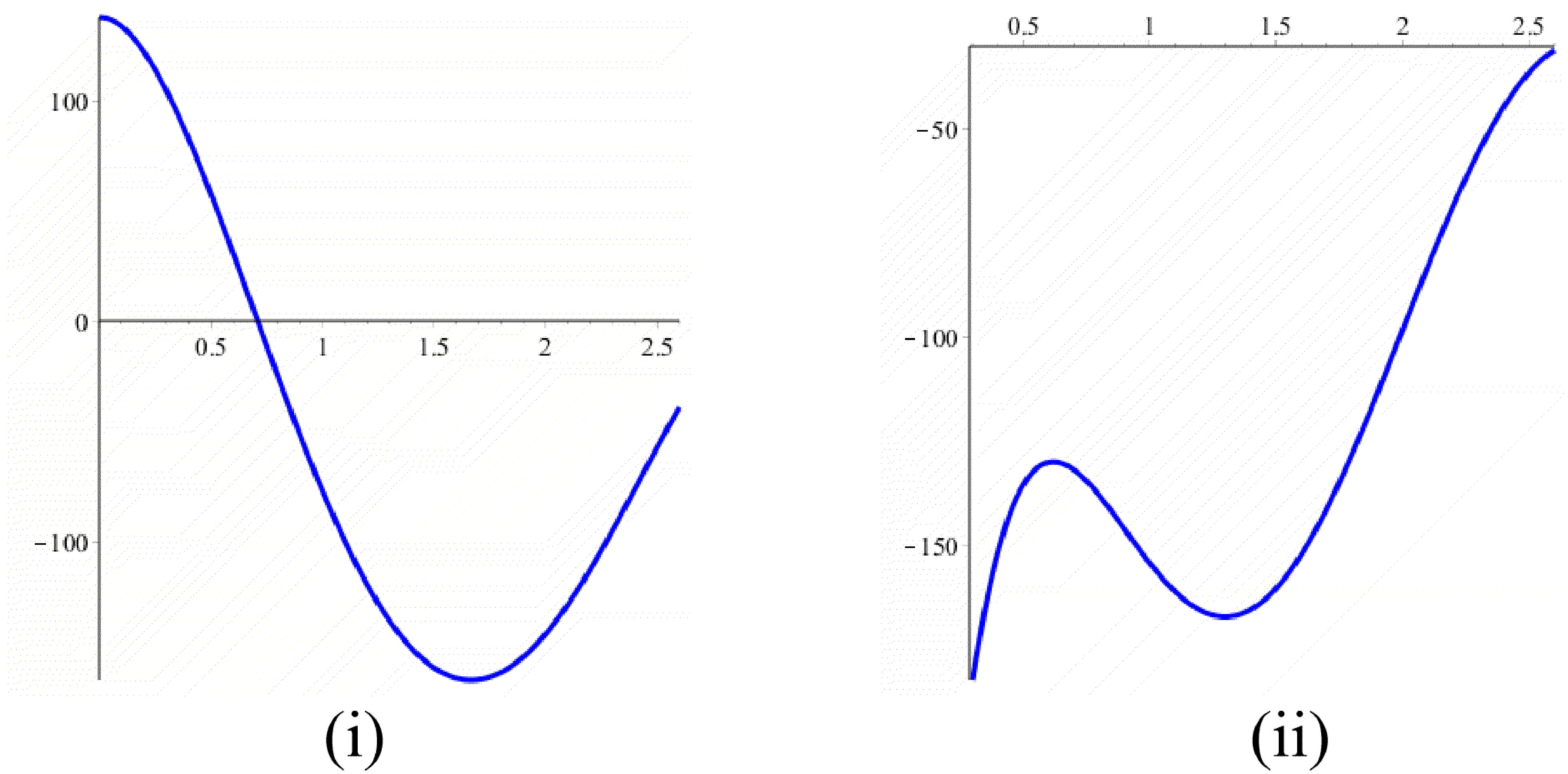}
  \caption{(i) The graph
of $-45u^{4}+280u^{3}-450u^{2}+138$ on $\left( 0,2.6\right) $. (ii) The
graph of $35u^{5}-345u^{4}+1230u^{3}-1902u^{2}+1242u-414$ on $\left(
0.3,2.6\right) .$}
  \label{fig3}
\end{figure}
\end{example}

In Example \ref{E9}, the problem was originally presented in \cite{ref8}.
Castro et al. \cite{ref8} proved that the corresponding bifurcation curve $S$
is $\subset $-shaped if $c>a>0$ and $b>0$. Obviously, This result is
evidently a specific case satisfying condition (\ref{4.1}) stated below.

\smallskip

\begin{example}
\label{E9}Consider (\ref{eq1}) with $f(u)=a+bu-ce^{-u}$ where 
\begin{equation}
\text{either (}b>0,\text{ }c>0\text{ and }a<c\text{), \ or \ (}b=0\text{ and 
}c>a>0\text{).}  \label{4.1}
\end{equation}%
Then the following statements (i)--(iii) hold:

\begin{itemize}
\item[(i)] If $a\leq 0$ and $b>0$, then the bifurcation curve $S$ is
monotone decreasing, starts from $(\hat{\lambda},\eta )$ and goes to $\left(
\kappa ,\infty \right) $ where $\hat{\lambda},\kappa \in (0,\infty )$.

\item[(ii)] If $a>0$ and $b>0$, then the bifurcation curve $S$ is $\subset $%
-shaped, starts from $(\hat{\lambda},\eta )$ and goes to $\left( \kappa
,\infty \right) $ where $\hat{\lambda},\kappa \in (0,\infty )$.

\item[(iii)] If $a>0$ and $b=0$, then the bifurcation curve $S$ is $\subset $%
-shaped, starts from $(\hat{\lambda},\eta )$ and goes to $\left( \infty
,\infty \right) $ where $\hat{\lambda}\in (0,\infty )$.
\end{itemize}

Next, we prove above statements. Clearly, $\beta _{2}=\infty $, $f^{\prime
}(u)=b+ce^{-u}>0$ and $f^{\prime \prime }(u)=-ce^{-u}<0$ for $u>0$. Then we
observe that%
\begin{equation*}
\text{(\ref{4.1}) holds \ if and only if \ (P}_{1}\text{)--(P}_{2}\text{)
and (H}_{4}\text{) hold.}
\end{equation*}%
We compute%
\begin{equation}
\lim\limits_{u\rightarrow \infty }\left[ f(u)-uf^{\prime }(u)\right]
=\lim\limits_{u\rightarrow \infty }\left( a-ce^{-u}-cue^{-u}\right) =a
\label{E9a}
\end{equation}%
and%
\begin{equation}
\lim\limits_{u\rightarrow \infty }g(u)=\lim\limits_{u\rightarrow \infty }%
\frac{a+bu-ce^{-u}}{u}=b.  \label{E9b}
\end{equation}%
By (\ref{4.1}), we see that $f(0^{+})=a-c\in (-\infty ,0)$, which implies
that%
\begin{equation*}
\lim\limits_{u\rightarrow 0^{+}}\sqrt{u}g(u)=\lim\limits_{u\rightarrow 0^{+}}%
\frac{f(u)}{\sqrt{u}}=-\infty .
\end{equation*}%
So by (\ref{E9b}) and Theorem \ref{T1}, then 
\begin{equation}
\hat{\lambda}\in (0,\infty )\text{ \ and \ }\left\{ 
\begin{array}{ll}
\kappa =\infty & \text{if }b=0\text{,} \\ 
\kappa \in (0,\infty ) & \text{if }b>0.%
\end{array}%
\right.  \label{E9c}
\end{equation}%
Since%
\begin{equation*}
\lim\limits_{u\rightarrow 0^{+}}u^{\frac{1}{3}}f(u)=0>-\infty ,
\end{equation*}%
and by Theorem \ref{T0}, we obtain that $G<0$. Next, we consider three cases.

\smallskip

Case 1. Assume that $a\leq 0$ and $b>0$. By (\ref{E9a}) and Remark \ref{R1}%
(ii), then (H$_{1}$) holds. By Theorems \ref{T1}--\ref{T2}(i) and (\ref{E9c}%
), the bifurcation curve $S$ is monotone decreasing, starts from $(\hat{%
\lambda},\eta )$ and goes to $\left( \kappa ,\infty \right) $. Thus the
statement (i) holds.

\smallskip

Case 2. Assume that $a>0$ and $b>0$. By (\ref{E9a}) and Remark \ref{R1}(ii),
then (H$_{2}$) holds. Since $G<0$, and by Theorems \ref{T1}--\ref{T2}(ii)
and (\ref{E9c}), the bifurcation curve $S$ is $\subset $-shaped, starts from 
$(\hat{\lambda},\eta )$ and goes to $\left( \kappa ,\infty \right) $. Thus
the statement (ii) holds.

\smallskip

Case 3. Assume that $a>0$ and $b=0$. By (\ref{E9a}) and Remark \ref{R1}(ii),
then (H$_{2}$) holds. Since $G<0$, and by Theorems \ref{T1}--\ref{T2}(ii)
and (\ref{E9c}), the bifurcation curve $S$ is $\subset $-shaped, starts from 
$(\hat{\lambda},\eta )$ and goes to $\left( \infty ,\infty \right) $. Thus
the statement (iii) holds.

\noindent By Cases 1--3, the proof is complete.
\end{example}

\section{Lemmas}

Since $F^{\prime }(u)=f(u)$, and by (P$_{1}$), we see that 
\begin{equation}
F^{\prime }(u)<0\text{ on }\left( 0,\beta _{1}\right) \cup \left( \beta
_{2},\infty \right) \text{, \ }F^{\prime }(\beta _{1})=0\text{ \ and \ }%
F^{\prime }(u)>0\text{ on }\left( \beta _{1},\beta _{2}\right) .  \label{dF}
\end{equation}%
Since $F(0^{+})=0$, and by (P$_{2}$), we observe that%
\begin{equation}
\left\{ u:u>0\text{, }f(u)>0\text{, }F(u)>F(s)\text{ \ for }0<s<u\right\}
=\left( \eta ,\beta _{2}\right)  \label{D}
\end{equation}%
To prove our main results, we use time-map techniques. The time-map formula
which we apply to study (\ref{eq1}) takes the form as follows:%
\begin{equation}
T(\alpha )\equiv \frac{1}{\sqrt{2}}\int_{0}^{\alpha }\frac{1}{\sqrt{F(\alpha
)-F(u)}}du=\sqrt{\lambda }\text{ \ for }\alpha \in \left( \eta ,\beta
_{2}\right) ,  \label{defT}
\end{equation}%
see Laetsch \cite{Laetsch}{\small . (}Note that it can be proved that $%
T(\alpha )$ is a twice differentiable function. The proof is easy but
tedious and hence we omit it.) So the positive solution $u$ of (\ref{eq1})
corresponds to%
\begin{equation*}
\left\Vert u\right\Vert _{\infty }=\alpha \text{ \ and \ }T(\alpha )=\sqrt{%
\lambda }.
\end{equation*}%
Then by definition of $S$ in (\ref{S}), we have that%
\begin{equation}
S=\left\{ \left( \lambda ,\alpha \right) :T(\alpha )=\sqrt{\lambda }\text{
for some }\alpha \in \left( \eta ,\beta _{2}\right) \text{ and }\lambda
>0\right\} .  \label{SS}
\end{equation}%
Thus, studying the shape of bifurcation curve $S$\ on the $(\lambda
,\left\Vert u\right\Vert _{\infty })$-plane is equivalent to studying the
shape of the time-map $T(\alpha )$ on $\left( \eta ,\beta _{2}\right) $.
Next, we let 
\begin{equation*}
\theta (u)\equiv 2F(u)-uf(u)\text{ \ for }u>0.
\end{equation*}%
Clearly, $\theta ^{\prime }(u)=f(u)-uf^{\prime }(u)$. Then we have the
following Lemmas \ref{LL1} and \ref{L1}.

\begin{lemma}
\label{LL1}Consider (\ref{eq1}). Then $\theta (0^{+})=0.$
\end{lemma}

\begin{proof}
By (P$_{1}$), we see that $\lim\limits_{u\rightarrow 0^{+}}uf(u)\leq 0$.
Assume that $\lim\limits_{u\rightarrow 0^{+}}uf(u)<0$. Then there exists $%
c>0 $ and $\delta >0$ such that $uf(u)<-c$ for $0<u<\delta $. It follows that%
\begin{equation*}
F(u)=\int_{0}^{u}f(t)dt<\int_{0}^{u}\frac{-c}{t}dt=-\infty \text{ \ for }%
0<u<\delta ,
\end{equation*}%
which is a contradiction by (P$_{2}$). Thus $\lim\limits_{u\rightarrow
0^{+}}uf(u)=0$. It follows that%
\begin{equation*}
\lim\limits_{u\rightarrow 0^{+}}\theta (u)=\lim\limits_{u\rightarrow 0^{+}} 
\left[ 2F(u)-uf(u)\right] =2F(0)=0
\end{equation*}%
The proof is complete.
\end{proof}

\begin{lemma}
\label{L1}Consider (\ref{eq1}). Then the statement (i)--(iii) hold:

\begin{itemize}
\item[(i)] Assume that (H$_{1}$) holds. Then $\theta ^{\prime }\left(
u\right) <0$ on $\left( 0,\beta _{2}\right) .$

\item[(ii)] Assume that (H$_{2}$) holds. Then%
\begin{equation}
\theta ^{\prime }\left( u\right) \left\{ 
\begin{array}{ll}
<0 & \text{for }0<u<\sigma , \\ 
=0 & \text{for }u=\sigma , \\ 
>0 & \text{for }\sigma <u<\beta _{2}.%
\end{array}%
\right.  \label{2.1}
\end{equation}%
Furthermore,

\begin{itemize}
\item[(a)] if $\theta (\beta _{2}^{-})\leq 0$, then $\theta \left( u\right)
<0$ on $(0,\beta _{2})$, see Figure \ref{fig5}(i); and

\item[(b)] if $\theta (\beta _{2}^{-})>0$, then there exists $\rho \in
(\sigma ,\beta _{2})$ such that $\theta (u)<0$ on $\left( 0,\rho \right) $, $%
\theta (\rho )=0$ and $\theta (u)>0$ on $\left( \rho ,\beta _{2}\right) $,
see Figure \ref{fig5}(ii).
\end{itemize}

\item[(iii)] Assume that (H$_{4}$) holds. Then

\begin{itemize}
\item[(a)] if $\theta ^{\prime }(\beta _{2}^{-})\leq 0$, then $\theta
^{\prime }\left( u\right) <0$ and $\theta \left( u\right) <0$ on $(0,\beta
_{2})$, see Figure \ref{fig5}(iii); and

\item[(b)] if $\theta ^{\prime }(\beta _{2}^{-})>0$, then there exist $%
0<\gamma <\rho <\beta _{2}$ such that%
\begin{equation}
\theta ^{\prime }\left( u\right) \left\{ 
\begin{array}{ll}
<0 & \text{for }0<u<\gamma , \\ 
=0 & \text{for }u=\gamma , \\ 
>0 & \text{for }\gamma <u<\beta _{2},%
\end{array}%
\right. \text{ \ and \ }\theta \left( u\right) \left\{ 
\begin{array}{ll}
<0 & \text{for }0<u<\rho , \\ 
=0 & \text{for }u=\rho , \\ 
>0 & \text{for }\rho <u<\beta _{2},%
\end{array}%
\right.  \label{3.6}
\end{equation}%
see Figure \ref{fig5}(iv).

\begin{figure}[h] 
  \centering
  \includegraphics[width=4.58in,height=4.69in,keepaspectratio]{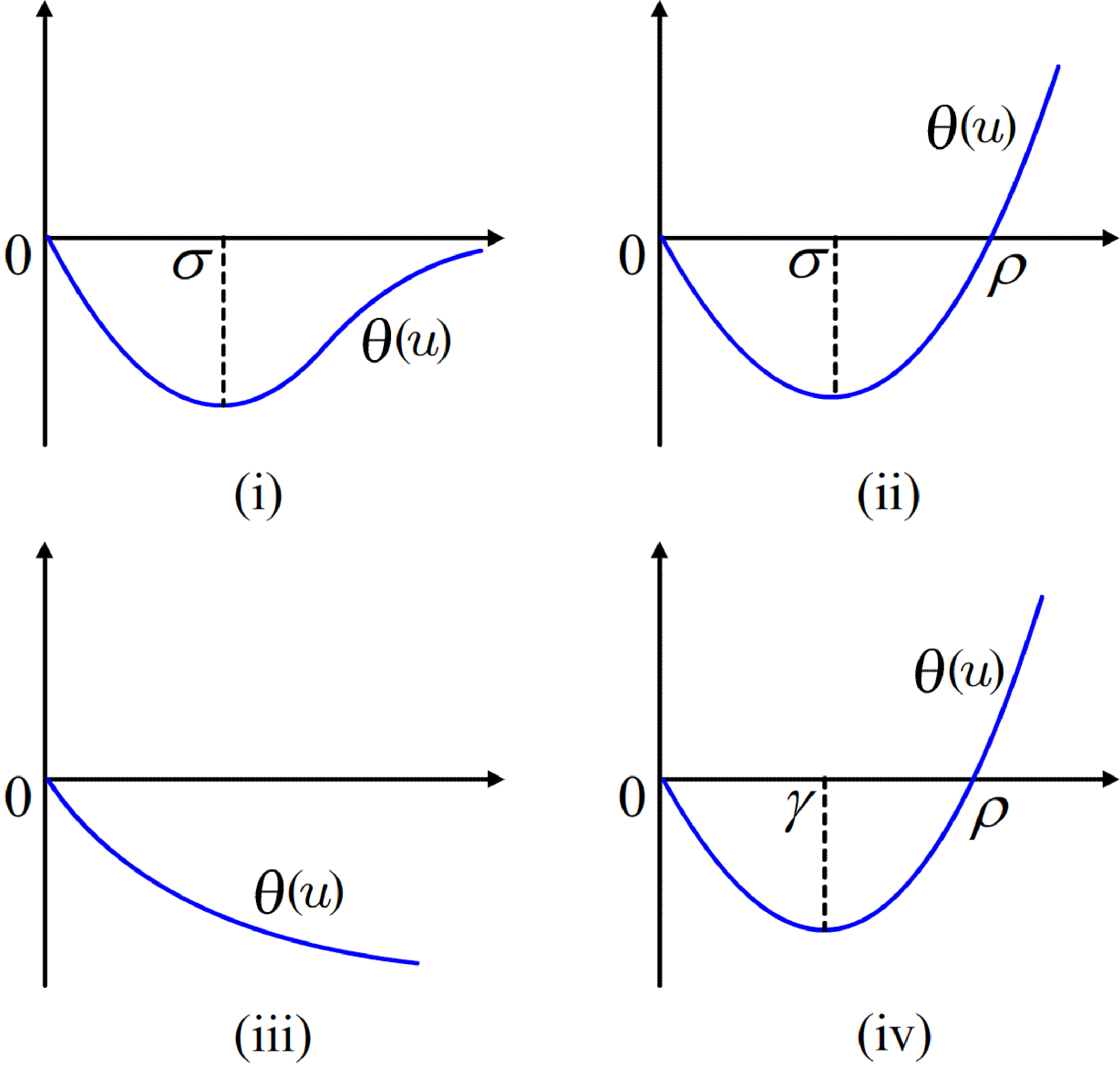}
  \caption{Graphs
of $\protect\theta \left( u\right) $. (i) (H$_{2}$) holds and $\protect%
\theta (\protect\beta _{2}^{-})\leq 0$. (ii) (H$_{2}$) holds and $\protect%
\theta (\protect\beta _{2}^{-})>0$. (ii) (H$_{4}$) holds and $\protect\theta %
^{\prime }(\protect\beta _{2}^{-})\leq 0$. (iii) (H$_{4}$) holds and $%
\protect\theta ^{\prime }(\protect\beta _{2}^{-})>0.$}
  \label{fig5}
\end{figure}
\end{itemize}
\end{itemize}
\end{lemma}

\begin{proof}
We compute%
\begin{equation}
g^{\prime }(u)=\frac{uf^{\prime }(u)-f(u)}{u^{2}}=-\frac{\theta ^{\prime
}\left( u\right) }{u^{2}}.  \label{L1a}
\end{equation}

(I) Assume that (H$_{1}$) holds. By (\ref{L1a}), then $\theta ^{\prime
}\left( u\right) <0$ on $\left( 0,\beta _{2}\right) $. So the statement (i)
holds.

(II) Assume that (H$_{2}$) holds. By (\ref{L1a}), then (\ref{2.1}) holds. If 
$\theta (\beta _{2}^{-})\leq 0$, by Lemma \ref{LL1} and (\ref{2.1}), we see
that $\theta \left( u\right) <0$ on $(0,\beta _{2})$. If $\theta (\beta
_{2}^{-})>0$, by Lemma \ref{LL1} and (\ref{2.1}), then there exists $\rho
\in (\sigma ,\beta _{2})$ such that $\theta (u)<0$ on $\left( 0,\rho \right) 
$, $\theta (\rho )=0$ and $\theta (u)>0$ on $\left( \rho ,\beta _{2}\right) $%
. So the statement (ii) holds.

(III) Assume that (H$_{4}$) holds. It follows that%
\begin{equation}
\theta ^{\prime \prime }\left( u\right) =-uf^{\prime \prime }(u)>0\text{ \
for }0<u<\beta _{2}.  \label{3.2}
\end{equation}%
Then we consider two cases.

\quad Case 1. Assume that $\theta ^{\prime }(\beta _{2}^{-})\leq 0$. By (\ref%
{3.2}), then $\theta ^{\prime }\left( u\right) <\theta ^{\prime }(\beta
_{2}^{-})\leq 0$ for $0<u<\beta _{2}$. It follows that $\theta \left(
u\right) <\theta (0^{+})=0$ for $0<u<\beta _{2}$ by Lemma \ref{LL1}. So the
statement (a) holds.

\quad Case 2. Assume that $\theta ^{\prime }(\beta _{2}^{-})>0$. By (P$_{1}$%
) and (H$_{4}$), then $f(0^{+})<0$ and $f^{\prime }(0^{+})>0$. It follows
that $\theta ^{\prime }\left( 0^{+}\right) \leq f(0^{+})<0$.\ Since $\theta
^{\prime }(\beta _{2}^{-})>0$, and by (\ref{3.2}), there exists $\gamma \in
(0,\beta _{2})$ such that 
\begin{equation}
\theta ^{\prime }\left( u\right) <0\text{ on }(0,\gamma )\text{, \ }\theta
^{\prime }\left( \gamma \right) =0\text{ \ and \ }\theta ^{\prime }\left(
u\right) >0\text{ on }(\gamma ,\beta _{2}).  \label{L1b}
\end{equation}%
If $\beta _{2}<\infty $, by (P$_{1}$), then $\theta (\beta _{2})=2F(\beta
_{2})>0$. So by Lemma \ref{LL1} and (\ref{L1b}), there exists $\rho \in
(\gamma ,\beta _{2})$ such that holds. If $\beta _{2}=\infty $, by Lemma \ref%
{LL1} and (\ref{3.2}), then $\lim\limits_{u\rightarrow \infty }\theta (u)>0$%
. So (\ref{3.6}) holds.

The proof is complete.
\end{proof}

\begin{lemma}
\label{L3}Consider (\ref{eq1}). Then the following statements (i)--(ii) hold:

\begin{itemize}
\item[(i)] If $\lim\limits_{u\rightarrow 0^{+}}g(u)\in (-\infty ,0]$, then $%
T(\eta ^{+})=\infty .$

\item[(ii)] If $\lim\limits_{u\rightarrow 0^{+}}u^{p}g(u)\in \lbrack -\infty
,0)$ for some $p\in (0,1)$, then $T(\eta ^{+})\in (0,\infty )$.
\end{itemize}
\end{lemma}

\begin{proof}
(I) Assume that $\lim\limits_{u\rightarrow 0^{+}}g(u)\in (-\infty ,0]$. By
L'H\"{o}pital's rule, we see that 
\begin{equation*}
\lim_{u\rightarrow 0^{+}}\frac{-F(u)}{u^{2}}=\lim_{u\rightarrow 0^{+}}\frac{%
-g(u)}{2}\in \lbrack 0,\infty ).
\end{equation*}%
There exist $c_{1}>0$ and $\delta _{1}\in (0,\eta )$ such that%
\begin{equation}
0<\frac{-F(u)}{u^{2}}<c_{1}\text{ \ for }0<u<\delta _{1}.  \label{3a}
\end{equation}%
Let $\varepsilon >0$. Since $F(\eta )=0$, there exists $\tilde{\delta}>0$
such that $0<F(\alpha )<\varepsilon $ for $\eta <\alpha <\eta +\tilde{\delta}
$. So by (\ref{D}) and (\ref{3a}), we observe that%
\begin{equation*}
0<F(\alpha )-F(u)<\varepsilon -F(u)<\varepsilon +c_{1}u^{2}
\end{equation*}%
for $0<u<\delta _{1}$ and $\eta <\alpha <\eta +\tilde{\delta}$. Then%
\begin{eqnarray*}
T(\eta ^{+}) &=&\lim_{\alpha \rightarrow \eta ^{+}}\frac{1}{\sqrt{2}}%
\int_{0}^{\alpha }\frac{1}{\sqrt{F(\alpha )-F(u)}}du\geq \lim_{\alpha
\rightarrow \eta ^{+}}\frac{1}{\sqrt{2}}\int_{0}^{\delta _{1}}\frac{1}{\sqrt{%
F(\alpha )-F(u)}}du \\
&\geq &\frac{1}{\sqrt{2}}\int_{0}^{\delta _{1}}\frac{1}{\sqrt{\varepsilon
+c_{1}u^{2}}}du=\frac{1}{\sqrt{2c_{1}}}\left[ \ln \left( \sqrt{c_{1}}\delta
_{1}+\sqrt{c_{1}\delta _{1}^{2}+\varepsilon }\right) -\ln \left( \sqrt{%
\varepsilon }\right) \right] .
\end{eqnarray*}%
Since $\varepsilon $ is arbitrary, we see that $T(\eta ^{+})=\infty $. Thus
the statement (i) holds.

(II) Assume that $\lim\limits_{u\rightarrow 0^{+}}u^{p}g(u)\in \lbrack
-\infty ,0)$ for some $p\in (0,1)$. Since $\eta \in (\beta _{1},\beta _{2})$%
, and by L'H\"{o}pital's rule, we see that%
\begin{equation*}
\lim_{u\rightarrow \eta ^{-}}\frac{-F(u)}{\eta -u}=\lim_{u\rightarrow \eta
^{-}}f(u)=f(\eta )>0.
\end{equation*}%
There exists $\delta _{2}\in (0,\eta )$ such that%
\begin{equation*}
\frac{-F(u)}{\eta -u}>\frac{f(\eta )}{2}>0\text{ \ for }\delta _{2}<u<\eta 
\text{.}
\end{equation*}%
So we obtain%
\begin{equation}
0<\int_{\delta _{2}}^{\eta }\frac{1}{\sqrt{-F(u)}}du<\sqrt{\frac{2}{f(\eta )}%
}\int_{\delta _{2}}^{\eta }\frac{1}{\sqrt{\eta -u}}du<\infty .  \label{1.2}
\end{equation}%
In addition, by L'H\"{o}pital's rule, we see that%
\begin{equation*}
\lim_{u\rightarrow 0^{+}}\frac{-F(u)}{u^{2-p}}=\lim_{u\rightarrow 0^{+}}%
\frac{-f(u)}{\left( 2-p\right) u^{1-p}}=\lim_{u\rightarrow 0^{+}}\frac{%
-u^{p}g(u)}{2-p}\in (0,\infty ]\text{.}
\end{equation*}%
Then there exist $c_{2}>0$ and $\delta _{3}\in (0,\delta _{2})$ such that%
\begin{equation*}
\frac{-F(u)}{u^{2-p}}>c_{2}\text{ \ for }0<u<\delta _{3}\text{.}
\end{equation*}%
Since $1<2-p<2$, we see that%
\begin{equation}
0<\int_{0}^{\delta _{3}}\frac{1}{\sqrt{-F(u)}}du<\frac{1}{\sqrt{c_{2}}}%
\int_{0}^{\delta _{3}}\frac{1}{\sqrt{u^{2-p}}}du<\infty \text{.}  \label{1.4}
\end{equation}%
By (\ref{1.2}), (\ref{1.4}) and continuity, then%
\begin{equation*}
0<T(\eta ^{+})=\int_{0}^{\delta _{3}}\frac{1}{\sqrt{-F(u)}}du+\int_{\delta
_{3}}^{\delta _{2}}\frac{1}{\sqrt{-F(u)}}du+\int_{\delta _{2}}^{\eta }\frac{1%
}{\sqrt{-F(u)}}du<\infty .
\end{equation*}%
Thus the statement (ii) holds. The proof is complete.
\end{proof}

\begin{lemma}
\label{L4}Consider (\ref{eq1}). Then the following statements (i)--(ii) hold.

\begin{itemize}
\item[(i)] Assume that $\beta _{2}=\infty $. Then%
\begin{equation*}
\left\{ 
\begin{array}{ll}
\lim\limits_{\alpha \rightarrow \infty }T(\alpha )=\infty & \text{if }%
\lim\limits_{u\rightarrow \infty }g(u)=0\text{,}\smallskip \\ 
\lim\limits_{\alpha \rightarrow \infty }T(\alpha )\in (0,\infty ) & \text{if 
}\lim\limits_{u\rightarrow \infty }g(u)\in (0,\infty )\text{,}\smallskip \\ 
\lim\limits_{\alpha \rightarrow \infty }T(\alpha )=0 & \text{if }%
\lim\limits_{u\rightarrow \infty }g(u)=\infty \text{.}%
\end{array}%
\right.
\end{equation*}

\item[(ii)] Assume that $\beta _{2}<\infty $. Then%
\begin{equation*}
\lim\limits_{\alpha \rightarrow \beta _{2}^{-}}T(\alpha )=\infty \text{ \ if 
}\lim\limits_{u\rightarrow \beta _{2}^{-}}\frac{f(u)}{\beta _{2}-u}=0,
\end{equation*}%
and%
\begin{equation*}
\lim\limits_{\alpha \rightarrow \beta _{2}^{-}}T(\alpha )\in (0,\infty )%
\text{ \ if }\lim\limits_{u\rightarrow \beta _{2}^{-}}\frac{f(u)}{\left(
\beta _{2}-u\right) \left[ -\ln \left( \beta _{2}-u\right) \right] ^{\tau }}%
=0\text{ for some }\tau >2.
\end{equation*}
\end{itemize}
\end{lemma}

\begin{proof}
(I) Assume that $\beta _{2}=\infty $. Then we consider three cases.

Case 1. Assume that $\lim\limits_{u\rightarrow \infty }g(u)=0$. For any $%
\varepsilon >0$, there exists $N_{1}>\eta $ such that $f(u)<\varepsilon u$
for $u\geq N_{1}$. Let $\alpha >2N_{1}$. Then 
\begin{equation}
F(\alpha )-F(u)=\int_{u}^{\alpha }f(t)dt<\varepsilon \int_{u}^{\alpha }tdt=%
\frac{\varepsilon }{2}\left( \alpha ^{2}-u^{2}\right) \text{ \ for }%
N_{1}\leq u<\alpha .  \label{5a}
\end{equation}%
Since $1/2>N_{1}/\alpha $, and by (\ref{5a}), we observe that%
\begin{eqnarray}
T(\alpha ) &>&\frac{1}{\sqrt{2}}\int_{N_{1}}^{\alpha }\frac{1}{\sqrt{%
F(\alpha )-F(u)}}du  \notag \\
&>&\frac{1}{\sqrt{\varepsilon }}\int_{N_{1}}^{\alpha }\frac{1}{\sqrt{\alpha
^{2}-u^{2}}}du=\frac{1}{\sqrt{\varepsilon }}\left( \frac{\pi }{2}-\arcsin 
\frac{N_{1}}{\alpha }\right)  \notag \\
&>&\frac{1}{\sqrt{\varepsilon }}\left( \frac{\pi }{2}-\arcsin \frac{1}{2}%
\right) =\frac{\pi }{3\sqrt{\varepsilon }}.  \label{5b}
\end{eqnarray}%
Since $\varepsilon $ is arbitrary, and by (\ref{5b}), we see that $%
\lim\limits_{\alpha \rightarrow \infty }T(\alpha )=\infty $.

Case 2. Assume that $\lim\limits_{u\rightarrow \infty }g(u)\in (0,\infty )$.
There exist $\varepsilon _{1}>0,$ $\varepsilon _{2}>0$ and $N_{2}>\eta $
such that $0<\varepsilon _{1}u<f(u)<\varepsilon _{2}u$ for $u>N_{2}$. Let $%
\alpha >2N_{2}$. By similar argument in Case 1, we obtain%
\begin{equation}
\lim\limits_{\alpha \rightarrow \infty }T(\alpha )>\frac{\pi }{3\sqrt{%
\varepsilon _{2}}}.  \label{5d}
\end{equation}%
By (P$_{1}$)--(P$_{2}$) and (\ref{dF}), we observe that%
\begin{equation*}
F^{\prime }(u)\left\{ 
\begin{array}{ll}
<0 & \text{for }0<u<\beta _{1}\text{,} \\ 
=0 & \text{for }u=\beta _{1}\text{,} \\ 
>0 & \text{for }u>\beta _{1}\text{,}%
\end{array}%
\right. \text{ \ and \ }F(u)\left\{ 
\begin{array}{ll}
<0 & \text{for }0<u<\eta \text{,} \\ 
=0 & \text{for }u=\eta \text{,} \\ 
>0 & \text{for }u>\eta \text{.}%
\end{array}%
\right.
\end{equation*}%
Then%
\begin{equation}
F(\alpha )-F(u)>F(2N_{2})-F(N_{2})>0\text{ \ for }0<u<N_{2}\text{.}
\label{5e}
\end{equation}%
In addition, we observe that%
\begin{equation}
F(\alpha )-F(u)=\int_{u}^{\alpha }f(t)dt>\varepsilon _{1}\int_{u}^{\alpha
}tdt=\frac{\varepsilon _{1}}{2}\left( \alpha ^{2}-u^{2}\right) \text{ \ for }%
N_{2}\leq u<\alpha .  \label{5c}
\end{equation}%
Since $1/2>N_{2}/\alpha $, and by (\ref{5e}) and (\ref{5c}), we see that 
\begin{eqnarray*}
T(\alpha ) &=&\frac{1}{\sqrt{2}}\int_{0}^{N_{2}}\frac{1}{\sqrt{F(\alpha
)-F(u)}}du+\frac{1}{\sqrt{2}}\int_{N_{2}}^{\alpha }\frac{1}{\sqrt{F(\alpha
)-F(u)}}du \\
&<&\frac{1}{\sqrt{2}}\int_{0}^{N_{2}}\frac{1}{\sqrt{F(2N_{2})-F(N_{2})}}du+%
\frac{1}{\sqrt{\varepsilon _{1}}}\int_{N_{2}}^{\alpha }\frac{1}{\sqrt{\alpha
^{2}-u^{2}}}du \\
&=&\frac{N_{2}}{\sqrt{2}\sqrt{F(2N_{2})-F(N_{2})}}+\frac{1}{\sqrt{%
\varepsilon _{1}}}\left( \frac{\pi }{2}-\arcsin \frac{N_{2}}{\alpha }\right)
\\
&<&\frac{N_{2}}{\sqrt{2}\sqrt{F(2N_{2})-F(N_{2})}}+\frac{1}{\sqrt{%
\varepsilon _{1}}}\frac{\pi }{2}<\infty .
\end{eqnarray*}%
So by (\ref{5d}), $\lim\limits_{\alpha \rightarrow \infty }T(\alpha )\in
(0,\infty ).$

Case 3. Assume that $\lim\limits_{u\rightarrow \infty }g(u)=\infty $. It
follows that $\lim\limits_{u\rightarrow \infty }f(u)=\infty $. So by L'H\"{o}%
pital's rule, then%
\begin{equation*}
\infty =\lim\limits_{u\rightarrow \infty }g(u)=\lim\limits_{u\rightarrow
\infty }f^{\prime }(u),
\end{equation*}%
which implies that there exists $N_{3}>\eta $ such that $f^{\prime }(u)>0$
for $u>N_{3}$. Let 
\begin{equation*}
L_{\alpha }(t)\equiv 1-\frac{F(\alpha t)}{F(\alpha )}\text{ \ for }0\leq
t\leq 1\text{ and }\alpha >N_{3}.
\end{equation*}%
Given $\alpha >N_{3}$. By (P$_{1}$), then%
\begin{equation}
L_{\alpha }(0)=L_{\alpha }(\frac{\eta }{\alpha })=1\text{ \ and \ }L_{\alpha
}^{\prime }(t)=-\frac{\alpha f(\alpha t)}{F(\alpha )}\left\{ 
\begin{array}{ll}
>0 & \text{for }0<t<\frac{\beta _{1}}{\alpha }, \\ 
=0 & \text{for }t=\frac{\beta _{1}}{\alpha }, \\ 
<0 & \text{for }\frac{\beta _{1}}{\alpha }<t<1.%
\end{array}%
\right.  \label{5f}
\end{equation}%
Since $N_{3}/\alpha >\eta /\alpha $, and by (\ref{5f}), we observe that%
\begin{equation}
L_{\alpha }(t)\geq L_{\alpha }(\frac{N_{3}}{\alpha })=1-\frac{F(N_{3})}{%
F(\alpha )}>0\text{ \ for }0\leq t\leq \frac{N_{3}}{\alpha }.  \label{5g}
\end{equation}%
Since%
\begin{equation*}
L_{\alpha }(1)=0\text{ \ and \ }L_{\alpha }^{\prime \prime }(t)=-\frac{%
\alpha ^{2}f^{\prime }(\alpha t)}{F(\alpha )}<0\text{ \ for }\frac{N_{3}}{%
\alpha }<t<1,
\end{equation*}%
we see that%
\begin{equation}
L_{\alpha }(t)>\frac{L_{\alpha }(\frac{N_{3}}{\alpha })}{1-\frac{N_{3}}{%
\alpha }}\left( 1-t\right) \text{ \ for }\frac{N_{3}}{\alpha }<t<1.
\label{5h}
\end{equation}%
By (\ref{5g}) and (\ref{5h}), then%
\begin{eqnarray}
T(\alpha ) &=&\frac{\alpha }{\sqrt{2}}\int_{0}^{1}\frac{1}{\sqrt{F(\alpha
)-F(\alpha t)}}dt\text{ \ (}t=\frac{u}{\alpha }\text{)}  \notag \\
&=&\frac{\alpha }{\sqrt{2F(\alpha )}}\int_{0}^{1}\frac{1}{\sqrt{1-F(\alpha
t)/F(\alpha )}}dt  \notag \\
&=&\frac{\alpha }{\sqrt{2F(\alpha )}}\left[ \int_{0}^{\frac{N_{3}}{\alpha }}%
\frac{1}{\sqrt{L_{\alpha }(t)}}dt+\int_{\frac{N_{3}}{\alpha }}^{1}\frac{1}{%
\sqrt{L_{\alpha }(t)}}dt\right]  \notag \\
&\leq &\frac{\alpha }{\sqrt{2F(\alpha )}}\left[ \int_{0}^{\frac{N_{3}}{%
\alpha }}\frac{1}{\sqrt{L_{\alpha }(\frac{N_{3}}{\alpha })}}dt+\sqrt{\frac{1-%
\frac{N_{3}}{\alpha }}{L_{\alpha }(\frac{N_{3}}{\alpha })}}\int_{\frac{N_{3}%
}{\alpha }}^{1}\frac{1}{\sqrt{1-t}}dt\right]  \notag \\
&=&\frac{\alpha }{\sqrt{2F(\alpha )}}\left[ \frac{N_{3}}{\alpha \sqrt{%
L_{\alpha }(\frac{N_{3}}{\alpha })}}+\frac{2}{\sqrt{L_{\alpha }(\frac{N_{3}}{%
\alpha })}}\left( 1-\frac{N_{3}}{\alpha }\right) \right] .  \label{5i}
\end{eqnarray}%
Since $\lim\limits_{\alpha \rightarrow \infty }F(\alpha )=\infty $, and by
L'H\"{o}pital's rule and (\ref{5g}), we see that%
\begin{equation*}
\lim_{\alpha \rightarrow \infty }\frac{F(\alpha )}{\alpha ^{2}}=\lim_{\alpha
\rightarrow \infty }\frac{g(\alpha )}{2}=\infty \text{ \ and \ }\lim_{\alpha
\rightarrow \infty }L_{\alpha }(\frac{N_{3}}{\alpha })=1\text{,}
\end{equation*}%
from which it follows that by (\ref{5i}), $\lim\limits_{\alpha \rightarrow
\infty }T(\alpha )=0$.

By Cases 1--3, the statement (i) holds.

(II) Assume that $\beta _{2}<\infty $. The proofs follow from \cite[Lemma 2.6%
]{Laetsch} and \cite[Theorem 5.2]{Lee2} after slight modifications. We omit
the proofs. The proof is complete.
\end{proof}

\begin{lemma}
\label{L5}Consider (\ref{eq1}). Assume that (H$_{2}$) and (H$_{3}$) hold.
Let 
\begin{equation*}
\xi \equiv \left\{ 
\begin{array}{ll}
\rho & \text{if }\theta (\beta _{2}^{-})>0,\smallskip \\ 
\beta _{2} & \text{if }\theta (\beta _{2}^{-})\leq 0,%
\end{array}%
\right.
\end{equation*}%
where $\rho $ is defined in Lemma \ref{L1}(ii). Then the following statement
(i)--(iii) hold:

\begin{itemize}
\item[(i)] If $\eta <\sigma $, then $T^{\prime }(\alpha )<0$ for $\eta
<\alpha \leq \sigma .$

\item[(ii)] 
\begin{equation}
T^{\prime \prime }(\alpha )+\frac{3+q(\alpha )}{\alpha }T^{\prime }(\alpha
)>0\text{ \ for }\sigma <\alpha <\xi ,  \label{5.10}
\end{equation}%
where $q(u)\equiv -u\theta ^{\prime }(u)/\theta (u)$. Moreover, $T(\alpha )$
has at most one critical point, a local minimum, on $(\sigma ,\xi )$.

\item[(iii)] If $\theta (\beta _{2}^{-})>0$, then $T^{\prime }(\alpha )>0$
for $\rho \leq \alpha <\beta _{2}.$
\end{itemize}
\end{lemma}

\begin{proof}
By (\ref{defT}), we compute%
\begin{equation}
T^{\prime }(\alpha )=\frac{1}{2\sqrt{2}\alpha }\int_{0}^{\alpha }\frac{%
2B(\alpha ,u)-A(\alpha ,u)}{\left[ B(\alpha ,u)\right] ^{3/2}}du  \label{dT}
\end{equation}%
and%
\begin{equation}
T^{\prime \prime }(\alpha )=\frac{1}{4\sqrt{2}\alpha ^{2}}\int_{0}^{\alpha }%
\frac{3A^{2}(\alpha ,u)-4A(\alpha ,u)B(\alpha ,u)-2B(\alpha ,u)C(\alpha ,u)}{%
B^{5/2}(\alpha ,u)}du  \label{dT2}
\end{equation}%
where $A(\alpha ,u)\equiv \alpha f(\alpha )-uf(u)$, $B(\alpha ,u)\equiv
F(\alpha )-F(u)$ and $C(\alpha ,u)\equiv \alpha ^{2}f^{\prime }(\alpha
)-u^{2}f^{\prime }(u)$. By (\ref{dF}), then%
\begin{equation}
B(\alpha ,u)=F(\alpha )-F(u)>0\text{ \ for }0<u<\alpha \text{ and }\eta
<\alpha <\beta _{2}.  \label{B}
\end{equation}

(I) Notice that%
\begin{equation}
2B(\alpha ,u)-A(\alpha ,u)=\theta (\alpha )-\theta (u).  \label{2BA}
\end{equation}%
By (\ref{2BA}) and Lemma \ref{L1}(ii), we see that 
\begin{equation}
2B(\alpha ,u)-A(\alpha ,u)=\theta (\alpha )-\theta (u)<0\text{ \ for }%
0<u<\alpha \leq \sigma .  \label{6.22}
\end{equation}%
If $\eta <\sigma $, by (\ref{dT}) and (\ref{6.22}), then $T^{\prime }(\alpha
)<0$ for $\eta <\alpha \leq \sigma $. So the statement (i) holds.

(II) Given $\alpha \in (\sigma ,\xi )$. By Lemma \ref{L1}(ii), then%
\begin{equation}
q(\alpha )=-\frac{\alpha \theta ^{\prime }(\alpha )}{\theta (\alpha )}>0%
\text{.}  \label{6.2}
\end{equation}%
By (\ref{dT}) and (\ref{dT2}), we have%
\begin{eqnarray}
&&T^{\prime \prime }(\alpha )+\frac{3+q(\alpha )}{\alpha }T^{\prime }(\alpha
)  \notag \\
&=&\frac{1}{2\sqrt{2}\alpha ^{2}}\int_{0}^{\alpha }\frac{\frac{3}{2}\left[
\theta (\alpha )-\theta (u)\right] ^{2}+B(\alpha ,u)\left[ Q(\alpha )-Q(u)%
\right] }{B^{5/2}}du  \notag \\
&\geq &\frac{1}{2\sqrt{2}\alpha ^{2}}\int_{0}^{\alpha }\frac{B(\alpha ,u)%
\left[ Q(\alpha )-Q(u)\right] }{B^{5/2}}du,  \label{TT}
\end{eqnarray}%
where $Q(u)\equiv u\theta ^{\prime }(u)+q(\alpha )\theta (u)$. By (\ref{6.2}%
), Lemmas \ref{LL1} and \ref{L1}(ii), we see that%
\begin{equation}
Q(0^{+})=Q(\alpha )=0\text{ \ and \ }Q(u)<0\text{ \ for }0<u\leq \sigma .
\label{5.3}
\end{equation}%
We assert that%
\begin{equation}
q^{\prime }(u)>0\text{ on }(\sigma ,\xi ).  \label{5.4}
\end{equation}%
By Lemma \ref{L1}(ii) and (\ref{5.4}), then%
\begin{equation}
Q(u)=\theta (u)\left[ q(\alpha )-q(u)\right] <0\text{ \ for }\sigma
<u<\alpha .  \label{5.5}
\end{equation}%
So by (\ref{5.3}) and (\ref{5.5}), we obtain%
\begin{equation}
Q(\alpha )-Q(u)>0\text{ \ for }0<u<\alpha .  \label{5.6}
\end{equation}%
It follows that (\ref{5.10}) holds by (\ref{B}), (\ref{TT}) and (\ref{5.6}).
In addition, if $\alpha _{1}$ is a critical point of $T(\alpha )$ on $\left(
\sigma ,\xi \right) $, then%
\begin{equation*}
T^{\prime \prime }(\alpha _{1})=T^{\prime \prime }(\alpha _{1})+\frac{%
3+q(\alpha _{1})}{\alpha _{1}}T^{\prime }(\alpha _{1})>0,
\end{equation*}%
which implies that $T(\alpha _{1})$ is a local minimum value. So $T(\alpha )$
has at most one critical point, a local minimum, on $\left( \sigma ,\xi
\right) $.

(III) Assume that $\theta (\beta _{2}^{-})>0$. By Lemma \ref{L1}(ii), $%
\theta (\alpha )-\theta (u)>0$ for $0<u<\alpha $ and $\rho \leq \alpha
<\beta _{2}$. So by (\ref{dT}) and (\ref{2BA}), then $T^{\prime }(\alpha )>0$
for $\rho \leq \alpha <\beta _{2}$. So the statement (iii) holds.

Finally, we prove the assertion (\ref{5.4}). We use the similar argument in 
\cite{Hung} to prove it. Since%
\begin{equation*}
q^{\prime }(u)=-\frac{\theta (u)\theta ^{\prime }(u)+u\theta (u)\theta
^{\prime \prime }(u)-u\theta ^{\prime }(u)\theta ^{\prime }(u)}{\theta
^{2}(u)},
\end{equation*}%
we see that $q^{\prime }(u)>0$ on $(\sigma ,\xi )$ if, and only if%
\begin{equation}
\frac{-\theta ^{\prime }(u)-u\theta ^{\prime \prime }(u)}{\theta ^{\prime
}(u)}<\frac{u\theta ^{\prime }(u)}{-\theta (u)}=q(u)\text{ \ for }\sigma
<u<\xi .  \label{6.1}
\end{equation}%
Let%
\begin{equation*}
I_{1}=\{u\in (\sigma ,\xi ):\theta ^{\prime }(u)+u\theta ^{\prime \prime
}(u)\geq 0\},
\end{equation*}%
\begin{equation*}
I_{2}=\{u\in (\sigma ,\xi ):\theta ^{\prime }(u)+u\theta ^{\prime \prime
}(u)<0\}.
\end{equation*}

Case 1. Assume that $u\in I_{1}$. By Lemma \ref{L1}(ii) and (\ref{6.2}), then%
\begin{equation*}
\frac{-\theta ^{\prime }(u)-u\theta ^{\prime \prime }(u)}{\theta ^{\prime
}(u)}\leq 0<q(u)\text{.}
\end{equation*}%
So by (\ref{6.1}), the assertion (\ref{5.4}) holds.

Case 2. Assume that $u\in I_{2}$. By Lemma \ref{L1}(ii) and (\ref{6.1}), we
see that $q^{\prime }(u)>0$ on $(\sigma ,\xi )$ if, and only if%
\begin{equation}
\frac{u\left[ \theta ^{\prime }(u)\right] ^{2}}{-\left[ \theta ^{\prime
}(u)+u\theta ^{\prime \prime }(u)\right] }>-\theta (u)\text{ \ for }\sigma
<u<\xi .  \label{6.3}
\end{equation}%
By (\ref{h3a}), we observe that%
\begin{equation*}
0\geq \left[ \frac{uf^{\prime }(u)}{f(u)}\right] ^{\prime }=\frac{f^{\prime
}(u)f(u)+uf^{\prime \prime }(u)f(u)-u\left[ f^{\prime }(u)\right] ^{2}}{%
f^{2}(u)}=\frac{f^{\prime }(u)\theta ^{\prime }(u)-\theta ^{\prime \prime
}(u)f(u)}{f^{2}(u)},
\end{equation*}%
which implies that%
\begin{equation}
f^{\prime }(u)\theta ^{\prime }(u)-\theta ^{\prime \prime }(u)f(u)\leq 0.
\label{6.4}
\end{equation}%
By Lemma \ref{L1}(ii) and (\ref{6.4}), we see that%
\begin{eqnarray}
&&-\left[ \theta ^{\prime }(u)+u\theta ^{\prime \prime }(u)\right] uf(u)-u%
\left[ \theta ^{\prime }(u)\right] ^{2}  \notag \\
&=&-uf(u)\theta ^{\prime }(u)-u^{2}f(u)\theta ^{\prime \prime }(u)-u\left[
\theta ^{\prime }(u)\right] ^{2}  \notag \\
&\leq &-uf(u)\theta ^{\prime }(u)-u^{2}f^{\prime }(u)\theta ^{\prime }(u)-u 
\left[ \theta ^{\prime }(u)\right] ^{2}  \notag \\
&=&\left[ -uf(u)-u^{2}f^{\prime }(u)-u\theta ^{\prime }(u)\right] \theta
^{\prime }(u)  \notag \\
&=&-2uf(u)\theta ^{\prime }(u)<0.  \label{6.5}
\end{eqnarray}%
Since $u\in I_{2}$, and by (\ref{6.5}), we see that%
\begin{equation*}
uf(u)-\frac{u\left[ \theta ^{\prime }(u)\right] ^{2}}{-\left[ \theta
^{\prime }(u)+u\theta ^{\prime \prime }(u)\right] }=\frac{-\left[ \theta
^{\prime }(u)+u\theta ^{\prime \prime }(u)\right] uf(u)-u\left[ \theta
^{\prime }(u)\right] ^{2}}{-\left[ \theta ^{\prime }(u)+u\theta ^{\prime
\prime }(u)\right] }<0<2F(u).
\end{equation*}%
So by (\ref{6.3}), the assertion (\ref{5.4}) holds.

The proof is complete.
\end{proof}

\begin{lemma}
\label{L5.5}Consider (\ref{eq1}) with $\beta _{2}=\infty $. Assume that (H$%
_{2}$) and (H$_{3}$) hold. Then the following statements (i)--(ii) hold:

\begin{itemize}
\item[(i)] If $\lim\limits_{u\rightarrow \infty }f(u)<\infty $, then $%
\lim\limits_{\alpha \rightarrow \infty }T(\alpha )=\infty .$

\item[(ii)] If $\lim\limits_{u\rightarrow \infty }f(u)=\infty $, then $%
\lim\limits_{u\rightarrow \infty }\theta (u)>0$.
\end{itemize}
\end{lemma}

\begin{proof}
(I) Assume that $\lim\limits_{u\rightarrow \infty }f(u)<\infty $. It implies
that $\lim\limits_{u\rightarrow \infty }g(u)=0$. So by Lemma \ref{L4}(i), we
obtain $\lim\limits_{\alpha \rightarrow \infty }T(\alpha )=\infty $. The
statement (i) holds.

(II) Assume that $\lim\limits_{u\rightarrow \infty }f(u)=\infty $. Since $%
g^{\prime }(\sigma )=0$, we see that $\sigma f^{\prime }(\sigma )=f(\sigma )$%
. So by (\ref{h3a}), then%
\begin{equation*}
\frac{uf^{\prime }(u)}{f(u)}<\frac{2\sigma f^{\prime }(2\sigma )}{f(2\sigma )%
}<\frac{\sigma f^{\prime }(\sigma )}{f(\sigma )}=1\text{ \ for }u>2\sigma .
\end{equation*}%
It follows that%
\begin{equation*}
\lim\limits_{u\rightarrow \infty }\theta ^{\prime
}(u)=\lim\limits_{u\rightarrow \infty }f(u)\left[ 1-\frac{uf^{\prime }(u)}{%
f(u)}\right] \geq \lim\limits_{u\rightarrow \infty }f(u)\left[ 1-\frac{%
2\sigma f^{\prime }(2\sigma )}{f(2\sigma )}\right] =\infty .
\end{equation*}%
So $\lim\limits_{u\rightarrow \infty }\theta (u)=\infty $. The statement
(ii) holds. The proof is complete.
\end{proof}

\begin{lemma}
\label{L6}Consider (\ref{eq1}). Assume that (H$_{4}$) holds. Then the
following statements (i)--(ii) hold:

\begin{itemize}
\item[(i)] If $G<0$ and $\theta ^{\prime }(\beta _{2}^{-})>0$, then $%
T(\alpha )$ is strictly decreasing and then strictly increasing on $\left(
\eta ,\beta _{2}\right) $.

\item[(ii)] If $G\geq 0$, then $T(\alpha )$ is strictly increasing on $%
\left( \eta ,\beta _{2}\right) $.
\end{itemize}
\end{lemma}

\begin{proof}
We divide this proof into the following four steps.

\textbf{Step 1. }We prove that $\alpha T^{\prime \prime }(\alpha
)+2T^{\prime }(\alpha )>0$ for $\eta <\alpha <\beta _{2}$. By (\ref{dT}) and
(\ref{dT2}), we have%
\begin{equation}
\alpha T^{\prime \prime }(\alpha )+2T^{\prime }(\alpha )=\frac{1}{4\sqrt{2}%
\alpha }\int_{0}^{\alpha }\frac{3\left( 2B-A\right) ^{2}+2B\left[ 2A-2B-C%
\right] }{B^{5/2}}du.  \label{4.5}
\end{equation}%
We observe that%
\begin{equation*}
\frac{\partial }{\partial u}\left[ 2A(\alpha ,u)-2B(\alpha ,u)-C(\alpha ,u)%
\right] =u^{2}f^{\prime \prime }(u)<0\text{ \ for }0<u<\beta _{2}.
\end{equation*}%
It follows that%
\begin{equation*}
2A(\alpha ,u)-2B(\alpha ,u)-C(\alpha ,u)>2A(\alpha ,\alpha )-2B(\alpha
,\alpha )-C(\alpha ,\alpha )=0
\end{equation*}%
for $0<u<\alpha <\beta _{2}$. So by (\ref{B}) and (\ref{4.5}), $\alpha
T^{\prime \prime }(\alpha )+2T^{\prime }(\alpha )>0$ for $\eta <\alpha
<\beta _{2}$.

\textbf{Step 2. }We prove that if $\theta ^{\prime }(\beta _{2}^{-})>0$,
then $T^{\prime }(\alpha )>0$ for $\rho <\alpha <\beta _{2}$. Assume that $%
\theta ^{\prime }(\beta _{2}^{-})>0$. By Lemma \ref{L1}(iii), then 
\begin{equation*}
2B(\alpha ,u)-A(\alpha ,u)=\theta (\alpha )-\theta (u)>0\text{ \ for }%
0<u<\alpha \text{ and }\rho <\alpha <\beta _{2}.
\end{equation*}%
So by (\ref{dT}), $T^{\prime }(\alpha )>0$ for $\rho <\alpha <\beta _{2}.$

\textbf{Step 3. }We prove the statement (i). If $\alpha _{1}$ is the
critical point of $T(\alpha )$ on $\left( \eta ,\beta _{2}\right) $, by Step
1, then%
\begin{equation*}
\alpha _{1}T^{\prime \prime }(\alpha _{1})=\alpha _{1}T^{\prime \prime
}(\alpha _{1})+2T^{\prime }(\alpha _{1})>0,
\end{equation*}%
which implies that $T(\alpha _{1})$ is a local minimum value. So 
\begin{equation}
T(\alpha )\text{ has at most one critical point, a local minimum, on }\left(
\eta ,\beta _{2}\right) .  \label{6a}
\end{equation}%
Assume that $G<0$ and $\theta ^{\prime }(\beta _{2}^{-})>0$. Then $T^{\prime
}(\eta ^{+})=G<0$. So by (\ref{6a}) and Step 2, we see that $T(\alpha )$ is
strictly decreasing and then strictly increasing on $\left( \eta ,\beta
_{2}\right) $. Thus the statement (i) holds.

\textbf{Step 4. }We prove the statement (ii). Assume that $G\geq 0$. We
consider two cases.

Case 1. Assume that $G>0$. Then $T^{\prime }(\eta ^{+})=G>0$. So by (\ref{6a}%
), we see that $T(\alpha )$ is strictly increasing on $\left( \eta ,\beta
_{2}\right) $.

Case 2. Assume that $G=0$. Then $T^{\prime }(\eta ^{+})=G=0$. Suppose there
exists $\delta >\eta $ such that $T^{\prime }(\alpha )<0$ on $\left( \eta
,\delta \right) $. Since $T^{\prime }(\eta ^{+})=0$, there exists $\alpha
_{2}\in \left( \eta ,\delta \right) $ such that $T^{\prime \prime }(\alpha
_{2})<0$. It follows that $\alpha _{2}T^{\prime \prime }(\alpha
_{2})+2T^{\prime }(\alpha _{2})<0$, which is a contradiction by Step 1. So
by (\ref{6a}), we observe that $T^{\prime }(\alpha )>0$ on $\left( \eta
,\beta _{2}\right) $.

By Cases 1--2, the statement (ii) holds.

The proof is complete.
\end{proof}

\begin{lemma}
\label{L7}Consider (\ref{eq1}). Assume that either ((H$_{2}$) and (H$_{3}$))
or ((H$_{2}$) and (H$_{4}$)) holds. Then $G<0$ if%
\begin{equation*}
\text{either }f(\eta )-\eta f^{\prime }(\eta )\leq 0\text{ \ or }%
\lim\limits_{u\rightarrow 0^{+}}u^{\frac{1}{3}}f(u)>-\infty .
\end{equation*}
\end{lemma}

\begin{proof}
We divide this proof into the next three steps.

\textbf{Step 1}. We prove that $G<0$ if $f(\eta )-\eta f^{\prime }(\eta
)\leq 0$. Since $\theta ^{\prime }(\eta )=f(\eta )-\eta f^{\prime }(\eta )<0$%
, and by Lemma \ref{L1}, we observe that $\theta (\eta )-\theta (u)<0$ for $%
0<u<\eta $. It follows that 
\begin{equation*}
G=\int_{0}^{\eta }\frac{-\eta f(\eta )-2F(u)+uf(u)}{[-F(u)]^{3/2}}%
du=\int_{0}^{\eta }\frac{\theta (\eta )-\theta (u)}{[-F(u)]^{3/2}}du<0.
\end{equation*}

\textbf{Step 2}. We prove that if $\theta ^{\prime }(\eta )>0$, there exists 
$\kappa _{1}\in (0,\eta )$ such that%
\begin{equation}
\int_{\kappa _{1}}^{\eta }\frac{\theta (\eta )-\theta (u)}{\left[ -F\left(
u\right) \right] ^{3/2}}du<\infty .  \label{5.09}
\end{equation}%
By L'H\"{o}pital's rule, we see that%
\begin{equation*}
\lim_{u\rightarrow \eta ^{-}}\frac{\theta (\eta )-\theta (u)}{\eta -u}%
=\theta ^{\prime }(\eta )>0.
\end{equation*}%
Then there exists $\delta _{1}\in (0,\eta )$ such that 
\begin{equation}
0<\theta (\eta )-\theta (u)<2\theta ^{\prime }(\eta )\left( \eta -u\right) 
\text{ \ for }\delta _{1}<u<\eta .  \label{7a}
\end{equation}%
Since $\theta ^{\prime }(\eta )>0$, we see that $f(\eta )>\eta f^{\prime
}(\eta )$. Then there exists $\kappa _{1}\in (\delta _{1},\eta )$ such that $%
f(u)>\eta f^{\prime }(\eta )$ for $\kappa _{1}<u<\eta .$ It follows that%
\begin{equation}
-F(u)=F(\eta )-F(u)=\int_{u}^{\eta }f(t)dt>\eta f^{\prime }(\eta )(\eta -u)%
\text{ \ for }\kappa _{1}<u<\eta .  \label{7b}
\end{equation}%
By (\ref{7a}) and (\ref{7b}), we obtain that%
\begin{equation*}
\int_{\kappa _{1}}^{\eta }\frac{\theta (\eta )-\theta (u)}{\left[ -F\left(
u\right) \right] ^{3/2}}du<\int_{\kappa _{1}}^{\eta }\frac{2\theta ^{\prime
}(\eta )\left( \eta -u\right) }{\left[ \eta f^{\prime }(\eta )(\eta -u)%
\right] ^{3/2}}du=\frac{2\theta ^{\prime }(\eta )}{\left[ \eta f^{\prime
}(\eta )\right] ^{3/2}}\int_{\kappa _{1}}^{\eta }\frac{1}{\sqrt{\eta -u}}%
du<\infty .
\end{equation*}%
Thus (\ref{5.09}) holds.

\textbf{Step 3}. We prove that $G<0$ if $\lim\limits_{u\rightarrow 0^{+}}u^{%
\frac{1}{3}}f(u)>-\infty $. We consider two cases.

Case 1. Assume that $\theta ^{\prime }(\eta )\leq 0$. By Step 1, we obtain
that $G<0$.

Case 2. Assume that $\theta ^{\prime }(\eta )>0$. By Lemma \ref{L1}, there
exists $\delta _{2}\in (0,\kappa _{1})$ such that 
\begin{equation}
\theta (\eta )-\theta (u)<\theta (\eta )-\theta (\delta _{2})<0\text{ \ for }%
0<u<\delta _{2}.  \label{5.03}
\end{equation}%
Since $\lim\limits_{u\rightarrow 0^{+}}u^{\frac{1}{3}}f(u)>-\infty $, and by
L'H\"{o}pital's rule, we see that%
\begin{equation*}
\lim_{u\rightarrow 0^{+}}\frac{F(u)}{u^{\frac{2}{3}}}=\frac{3}{2}%
\lim_{u\rightarrow 0^{+}}u^{\frac{1}{3}}f(u)\in (-\infty ,0].
\end{equation*}%
Then there exist $\varepsilon >0$ and $\kappa _{2}\in (0,\delta _{2})$ such
that 
\begin{equation}
-\varepsilon <\frac{F(u)}{u^{\frac{2}{3}}}<0\text{ \ for }0<u<\kappa _{2}.
\label{5.04}
\end{equation}%
By (\ref{5.03}) and (\ref{5.04}), then%
\begin{equation*}
\int_{0}^{\kappa _{2}}\frac{\theta (\eta )-\theta (u)}{\left[ -F\left(
u\right) \right] ^{3/2}}du<\frac{\theta (\eta )-\theta (\delta _{2})}{%
\varepsilon ^{\frac{3}{2}}}\int_{0}^{\kappa _{2}}\frac{1}{u}du=-\infty \text{
}
\end{equation*}%
So by Step 2, 
\begin{eqnarray*}
G &=&\int_{0}^{\eta }\frac{-\eta f(\eta )-2F(u)+uf(u)}{[-F(u)]^{3/2}}%
du=\int_{0}^{\eta }\frac{\theta (\eta )-\theta (u)}{[-F(u)]^{3/2}}du \\
&=&\int_{0}^{\kappa _{2}}\frac{\theta (\eta )-\theta (u)}{\left[ -F\left(
u\right) \right] ^{3/2}}du+\int_{\kappa _{2}}^{\kappa _{1}}\frac{\theta
(\eta )-\theta (u)}{\left[ -F\left( u\right) \right] ^{3/2}}du+\int_{\kappa
_{1}}^{\eta }\frac{\theta (\eta )-\theta (u)}{\left[ -F\left( u\right) %
\right] ^{3/2}}du \\
&=&-\infty .
\end{eqnarray*}

The proof is complete.
\end{proof}

\section{Proofs of Main Results}

\smallskip

\begin{proof}[Proof of Theorem \protect\ref{T1}]
Notice that $\hat{\lambda}=\left[ T(\eta ^{+})\right] ^{2}$ and $\kappa =%
\left[ T(\beta _{2}^{-})\right] ^{2}$. So by (\ref{defT}), (\ref{SS}) and
Lemmas \ref{L3}--\ref{L4}, the proof is complete.
\end{proof}

\smallskip

\begin{proof}[Proof of Theorem \protect\ref{T2}]
We divide this proof into the following three steps.

\textbf{Step 1. }We prove the statement (i). Assume that (H$_{1}$) holds. By
Lemma \ref{L1}(i), then $\theta \left( \alpha \right) -\theta \left(
u\right) <0$ for $0<u<\alpha <\beta _{2}$. So by (\ref{dT}) and (\ref{2BA}), 
$T^{\prime }(\alpha )<0$ for $\eta <\alpha <\beta _{2}$. It follows that the
bifurcation curve $S$ is monotone decreasing by (\ref{defT}) and (\ref{SS}).

\smallskip

\textbf{Step 2. }We prove the statement (ii) if (H$_{2}$) and (H$_{3}$)
hold, and $\beta _{2}<\infty $. Assume that (H$_{2}$) and (H$_{3}$) hold,
and $\beta _{2}<\infty $. Since $f(\beta _{2})=0$, we see that $\theta
(\beta _{2})=2F(\beta _{2})>0$. So by Lemma \ref{L5}, we obtain that%
\begin{equation}
\left\{ 
\begin{array}{l}
T(\alpha )\text{ has at most one critical point, a local minimum, on }%
(\sigma ,\rho ),\smallskip \\ 
T^{\prime }(\alpha )>0\text{ on }[\rho ,\beta _{2}).%
\end{array}%
\right.  \label{2.1a}
\end{equation}%
Then we consider four cases.

Case 1. Assume that $G<0$ and $\eta <\sigma $. By Lemma \ref{L5}(i), $%
T^{\prime }(\alpha )<0$ for $\eta <\alpha \leq \sigma $. So by (\ref{2.1a}), 
$T(\alpha )$ is strictly decreasing and then strictly increasing on $\left(
\eta ,\beta _{2}\right) $.

Case 2. Assume that $G<0$ and $\eta \geq \sigma $. Since $T^{\prime }(\eta
^{+})=G<0$, and by (\ref{2.1a}), $T(\alpha )$ is strictly decreasing and
then strictly increasing on $\left( \eta ,\beta _{2}\right) $.

Case 3. Assume that $G>0$ and $\eta <\sigma $. By Lemma \ref{L5}(i), $%
T^{\prime }(\alpha )<0$ for $\eta <\alpha \leq \sigma $. It follows that $%
G=T^{\prime }(\eta ^{+})\leq 0$. It is a contradiction.

Case 4. Assume that $G>0$ and $\eta \geq \sigma $. Since $T^{\prime }(\eta
^{+})=G>0$, and by (\ref{2.1a}), $T(\alpha )$ is strictly increasing on $%
\left( \eta ,\beta _{2}\right) $.

\noindent By Cases 1--4, (\ref{defT}) and (\ref{SS}), the statement (ii)
holds.

\smallskip

\textbf{Step 3. }We prove the statement (ii) if (H$_{2}$) and (H$_{3}$)
hold, and $\beta _{2}=\infty $. Assume that (H$_{2}$) and (H$_{3}$) hold,
and $\beta _{2}=\infty $. We consider two cases.

Case 1. Assume that $\lim\limits_{u\rightarrow \infty }f(u)=\infty $. By
Lemma \ref{L5.5}(ii), then $\theta (\beta
_{2}^{-})=\lim\limits_{u\rightarrow \infty }\theta (u)>0$. So by Lemma \ref%
{L5}, then (\ref{2.1a}) holds. Then by similar argument in Step 2, the
statement (ii) holds.

Case 2. Assume that $\lim\limits_{u\rightarrow \infty }f(u)<\infty $. If $%
\theta (\beta _{2}^{-})=\lim\limits_{u\rightarrow \infty }\theta (u)>0$, by
similar argument in Step 2, then the statement (ii) holds. If $\theta (\beta
_{2}^{-})=\lim\limits_{u\rightarrow \infty }\theta (u)\leq 0$, by Lemmas \ref%
{L5} and \ref{L5.5}(i), then%
\begin{equation*}
\left\{ 
\begin{array}{l}
T(\alpha )\text{ has at most one critical point, a local minimum, on }%
(\sigma ,\infty ),\smallskip \\ 
\lim\limits_{\alpha \rightarrow \infty }T(\alpha )=\infty .%
\end{array}%
\right.
\end{equation*}%
So by similar argument in Step 2, then the statement (ii) holds.

\textbf{Step 4. }We prove the statement (iii) if (H$_{2}$) and (H$_{4}$)
hold. Since (H$_{2}$) holds, we see that $\theta ^{\prime }(\beta
_{2}^{-})>0 $. If $G<0$, by Lemma \ref{L6}, then $T(\alpha )$ is strictly
decreasing and then strictly increasing on $\left( \eta ,\beta _{2}\right) $%
; and if $G\geq 0$, then $T(\alpha )$ is strictly increasing on $\left( \eta
,\beta _{2}\right) $. So by (\ref{defT}) and (\ref{SS}), the statement (iii)
holds.

The proof is complete.
\end{proof}

\smallskip

\begin{proof}[Proof of Corollary \protect\ref{C0}]
Assume that $\beta _{2}<\infty $ and (H$_{4}$) holds. Since $\beta
_{2}<\infty $ and $f(\beta _{2})=0$, we see that $\theta (\beta
_{2})=2F(\beta _{2})>0$. So by Lemma \ref{L1}(iii), there exists $\gamma \in
(0,\beta _{2})$ such that $\theta ^{\prime }(u)<0$ on $\left( 0,\gamma
\right) $, $\theta ^{\prime }(\gamma )=0$ and $\theta ^{\prime }(u)>0$ on $%
\left( \gamma ,\beta _{2}\right) $. Then we observe that%
\begin{equation*}
g^{\prime }(u)=\frac{-\theta ^{\prime }(u)}{u^{2}}\left\{ 
\begin{array}{ll}
>0 & \text{for }0<u<\gamma , \\ 
=0 & \text{for }u=\gamma , \\ 
<0 & \text{for }\gamma <u<\beta _{2},%
\end{array}%
\right.
\end{equation*}%
which implies that (H$_{2}$) holds. So by Theorem \ref{T2}, the bifurcation
curve $S$ is $\subset $-shaped if $G<0$, and is monotone increasing if $%
G\geq 0$. The proof is complete.
\end{proof}

\smallskip

\begin{proof}[Proof of Theorem \protect\ref{T0}]
By Lemma \ref{L7}, then $G<0$ if 
\begin{equation*}
\text{either }f(\eta )-\eta f^{\prime }(\eta )\leq 0\text{ \ or }%
\lim\limits_{u\rightarrow 0^{+}}u^{\frac{1}{3}}f(u)>-\infty .
\end{equation*}%
If $f(0^{+})>-\infty $, then%
\begin{equation*}
\lim\limits_{u\rightarrow 0^{+}}u^{\frac{1}{3}}f(u)=0>-\infty ,
\end{equation*}%
which implies that $G<0$. The proof is complete.
\end{proof}

\smallskip

\begin{proof}[Proof of Corollary \protect\ref{C1}]
Assume that $f(0^{+})>-\infty $. By Theorem \ref{T0}, then $G<0$. If either
((H$_{2}$) and (H$_{3}$) hold), or ($\beta _{2}<\infty $ and (H$_{4}$)
holds), by Theorem \ref{T2} and Corollary \ref{C0}, then the bifurcation
curve $S$ is $\subset $-shaped. The proof is complete.
\end{proof}

\smallskip

\begin{proof}[Proof of Corollary \protect\ref{C2}]
Assume that $\beta _{2}=\infty $ and $f^{\prime \prime }(u)>0$ on $\left(
0,\infty \right) $. Then $\theta ^{\prime \prime }(u)=-uf^{\prime \prime
}(u)<0$. We assert that $\theta ^{\prime }(0^{+})\leq 0$. It follows that $%
\theta ^{\prime }(u)<\theta ^{\prime }(0^{+})\leq 0$ for $0<u<\beta _{2}$.
So we see that%
\begin{equation*}
g^{\prime }(u)=\frac{uf^{\prime }(u)-f(u)}{u^{2}}=\frac{-\theta ^{\prime }(u)%
}{u^{2}}>0\text{ \ for }0<u<\beta _{2},
\end{equation*}%
which implies that (H$_{1}$) holds. Then by Theorem \ref{T2}(i), the
bifurcation curve $S$ is monotone decreasing. In addition, since $\beta
_{2}=\infty $ and $f^{\prime \prime }(u)>0$ on $\left( 0,\infty \right) $,
we observe that 
\begin{equation*}
\lim\limits_{u\rightarrow \infty }f(u)=\infty \text{ \ and \ }%
\lim\limits_{u\rightarrow \infty }f^{\prime }(u)\in (0,\infty ].
\end{equation*}%
It follows that%
\begin{equation*}
\lim\limits_{u\rightarrow \infty }g(u)=\lim\limits_{u\rightarrow \infty
}f^{\prime }(u)\in (0,\infty ].
\end{equation*}%
So by Theorem \ref{T1}, we obtain (\ref{c2a}).

Next, we prove that $\theta ^{\prime }(0^{+})\leq 0$. Since $f^{\prime
\prime }(u)>0$ on $\left( 0,\infty \right) $, we observe that $-\infty
<f(0^{+})\leq 0$. Suppose $\lim_{u\rightarrow 0^{+}}uf^{\prime }(u)<0$.
There exist $\delta ,\varepsilon >0$ such that $uf^{\prime }(u)<-\varepsilon 
$ for $0<u\leq \delta $. Then 
\begin{equation*}
-\infty <f(\delta )-f(0^{+})=\int_{0}^{\delta }f^{\prime }(t)dt<-\varepsilon
\int_{0}^{\delta }\frac{1}{t}dt=-\infty ,
\end{equation*}
which is a contradiction. So $\lim_{u\rightarrow 0^{+}}uf^{\prime }(u)\geq 0$%
. Then%
\begin{equation*}
\theta ^{\prime }(0^{+})=\lim_{u\rightarrow 0^{+}}\left[ f(u)-uf^{\prime }(u)%
\right] \leq 0.
\end{equation*}%
The proof is complete.
\end{proof}

\section{Appendix}

In this section, we point out the mistake of the proof of Theorem \ref{RT2}
in \cite{ref8}. In \cite{ref8}, authors proved that 
\begin{equation}
T^{\prime }(\alpha )<0\text{ \ for }0<\alpha \leq \gamma \text{ \ if (\ref%
{3.6}) holds.}  \label{e1}
\end{equation}%
Notice that (\ref{e1}) is an important conclusion in demonstrating that the
bifurcation curve is $\subset $-shaped in \cite{ref8}. However, we give a
example to illustrate that%
\begin{equation}
T(\alpha )\text{ may not be well-defined for }0<\alpha \leq \gamma \text{ \
if (\ref{3.6}) holds.}  \label{e2}
\end{equation}%
It follows that (\ref{e1}) is not completely correct.

We consider the function $f\in C^{2}(0,\infty )$ satisfying that%
\begin{equation*}
f(u)=-u^{2}+\frac{21}{10}u-1\text{ \ for }0<u<1.02.
\end{equation*}%
Clearly, (P$_{1}$) and (P$_{2}$) hold. We compute that%
\begin{equation*}
F(u)=-\frac{1}{3}u^{3}+\frac{21}{20}u^{2}-u\text{ \ and \ }f^{\prime
}(u)=-2u+\frac{21}{10}\text{ \ for }0<u<1.02.
\end{equation*}%
Since%
\begin{equation*}
\theta ^{\prime }(u)=f(u)-uf^{\prime }(u)=\left( u-1\right) \left(
u+1\right) \left\{ 
\begin{array}{ll}
<0 & \text{for }0<u<1, \\ 
=0 & \text{for }u=1, \\ 
>0 & \text{for }1<u<1.02,%
\end{array}%
\right.
\end{equation*}%
we see that (\ref{3.6}) holds and $\gamma =1$. Then we observe that $F(u)<0$
for $0<u\leq \gamma =1$, see Figure \ref{fig6}. 

\begin{figure}[h] 
  \centering
  \includegraphics[width=3.05in,height=2.94in,keepaspectratio]{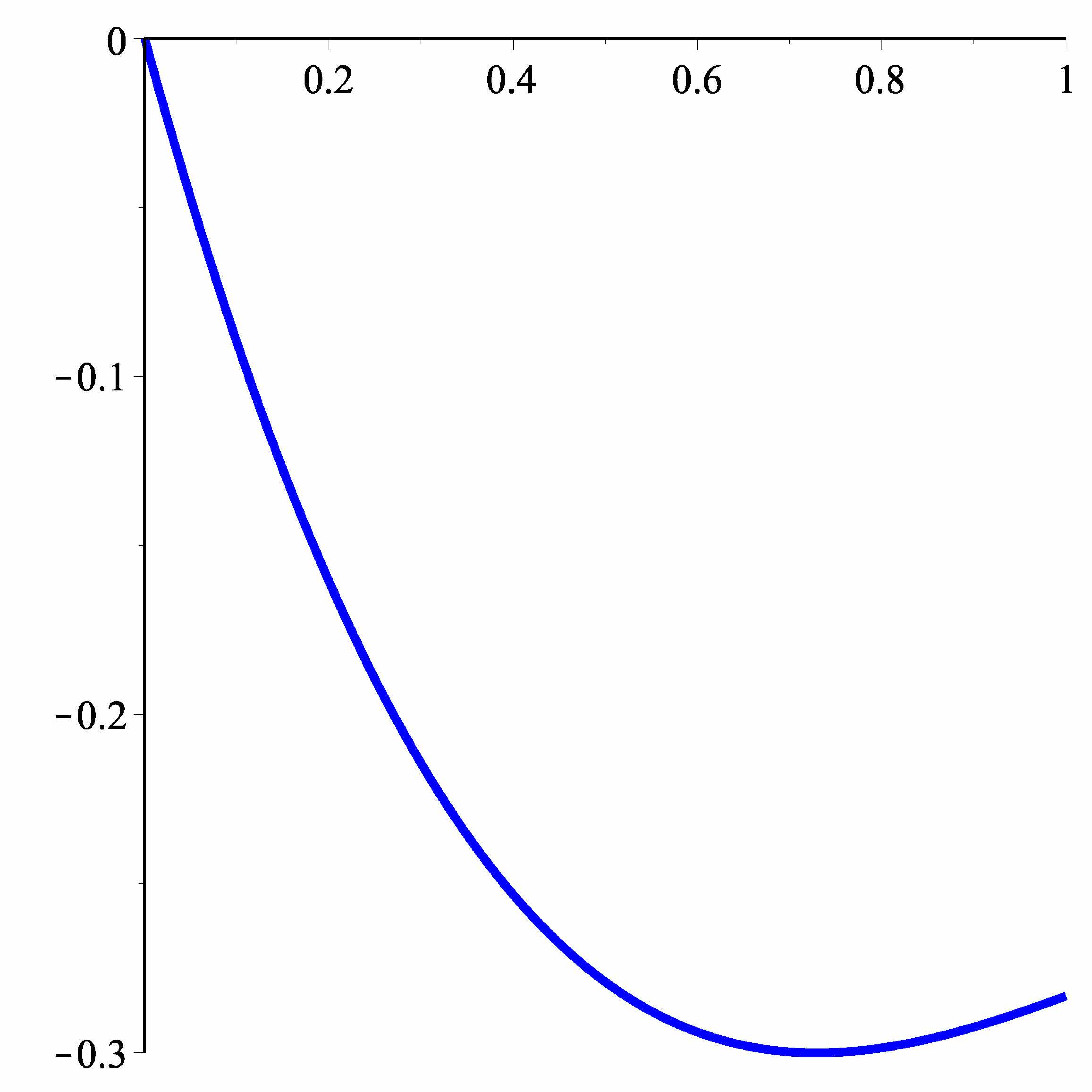}
  \caption{The grpah of $F(u)=-\frac{1}{3}u^{3}+\frac{21}{20}%
u^{2}-u $ on $[0,\protect\gamma ]$.}
  \label{fig6}
\end{figure}
So by the definition of $T$, we know that $T(\alpha )$ is not
well-defined for $0<\alpha \leq \gamma =1.$ Thus the proof of Theorem \ref%
{RT2} is not complete under the case $\gamma <\eta .$

\section{Statements and Declarations}

\begin{description}
\item[Funding] The authors declare that no funds, grants, or other support
were received during the preparation of this manuscript.

\item[Competing Interests] The authors have no relevant financial or
non-financial interests to disclose.

\item[Author Contributions] This article has only one author. The study
conception and design, material preparation, data collection and analysis,
as well as the writing of the manuscript, were all carried out solely by the
author.
\end{description}

\bigskip

\end{document}